\documentclass[preprint,12pt,3p]{elsarticle-no-natbib}

\usepackage{amssymb}
\usepackage{amsthm}

\usepackage[backend=bibtex,%
style=numeric,%
bibencoding=ascii
eprint=false,isbn=false,doi=false]{biblatex}

\renewbibmacro{in:}{%
   \ifentrytype{article}{}{\printtext{\bibstring{in}\intitlepunct}}}
  
\usepackage{amsmath}
\usepackage{amsfonts}
\usepackage{graphicx}
\usepackage{caption}
\usepackage[labelformat=simple]{subcaption}

\usepackage{color}
\usepackage[shortlabels]{enumitem}
\usepackage{pdfpages}
\usepackage{harpoon}
\usepackage[colorlinks=true,allcolors=blue]{hyperref}
\newtheorem{theorem}{Theorem}[section]
\newtheorem{proposition}[theorem]{Proposition}
\newtheorem{lemma}[theorem]{Lemma}

\newtheorem{corollary}[theorem]{Corollary}

\newtheorem*{theorem-densemain}{Theorem \ref{densemain}}
\theoremstyle{definition}
\newtheorem{definition}[theorem]{Definition}

\newtheorem{observation}[theorem]{Observation}

\newtheorem{example}[theorem]{Example}

\theoremstyle{remark}

\numberwithin{equation}{section}
	{\vskip\topsep\noindent\textsc{#1.\/}}%
	{\par\vskip\topsep}%
\theoremstyle{definition}
\newtheorem{case}{Case}
\newtheorem{subcase}{Case}[case]
\newtheorem{subsubcase}{Case}[subcase]
\definecolor{dkgreen}{rgb}{0,0.7,0.1}
\long\def\ignore#1{}
\def\usg#1{#1_{\fam0 us}}
\def\sg#1{#1_{\fam0 s}}
\def\realline{\hbox to \hsize}% plain TeX \line
\def\cC{\mathcal{C}}
\def\cT{{\mathcal{T}}}
\def\cA{\mathcal{A}}
\def\mw{\!\cdot\!}% merge walks 
\let\al\alpha
\let\ga\gamma
\let\ep\varepsilon
\let\la\lambda
\let\nr\rho
\let\nw\omega
\def\hf{{\textstyle\frac12}}

\def\indeg{{\fam0 indeg}}
\def\outdeg{{\fam0 outdeg}}

\let\fw\gamma% following half-arc in circuit decomposition
\def\iv{^{-1}}% inverse
\let\mate\omega% other end of edge
\let\incv\psi% vertex to which half-edge/arc incident

\def\cW{\mathcal{W}}

\def\mR{\mathbb{R}}
\def\mZ{\mathbb{Z}}

\def\only#1{{#1{}0}}
\def\also#1{{#1{}1}}
\def\any#1{}% remove this notation
\let\ov\overline

\def\siran#1{\v{S}ir\'{a}\v{n}}% use as "\siran."
\def\nebesky#1{Nebesk\'{y}}% use as "\nebesky."

\journal{To be decided}
\newtoks\datetoshow\datetoshow={\today}
\def\setdatetoshow#1{\global\datetoshow={#1}}
\makeatletter
\def\ps@pprintTitleNoSubTo{%
     \let\@oddhead\@empty
     \let\@evenhead\@empty
     \def\@oddfoot{\footnotesize\itshape\hfill\the\datetoshow}%
     \let\@evenfoot\@oddfoot}
\let\ps@pprintTitle\ps@pprintTitleNoSubTo% uncomment for no "Sub to ..."
\makeatother

\begin{document}

\begin{frontmatter}

\title{Maximum genus embeddings of dense eulerian graphs with specified faces}

\author[label1, label3]{M. N. Ellingham\corref{cor1}}

\address[label1]{Department of Mathematics, 1326 Stevenson Center,
Vanderbilt University,
Nashville, Tennessee 37240}
\fntext[label3]{Supported by Simons Foundation awards 429625 and MPS-TSM-00002760}

\ead{mark.ellingham@vanderbilt.edu} \ead[url]{https://math.vanderbilt.edu/ellingmn/}

\author[label2]{Joanna A. Ellis-Monaghan}

\address[label2]{Korteweg-de Vries Institute for Mathematics, University of Amsterdam, Science Park 105-107, 1098 XH Amsterdam, the Netherlands}
\ead{jellismonaghan@gmail.com}
\ead[url]{https://sites.google.com/site/joellismonaghan/}

\setdatetoshow{22 September 2024}% Fill in date and uncomment

\begin{abstract}
We give a density condition for when, subject to a necessary parity condition, an eulerian graph or digraph may be cellularly embedded in an orientable surface so that it has exactly two faces, each bounded by an euler circuit, one of which may be specified in advance.  More generally, suppose that every vertex in an $n$-vertex eulerian digraph (loops and multiple arcs allowed) has at least $(4n+2)/5$ neighbors, and specify any decomposition of the arcs into disjoint directed circuits (closed trails).  We show that such a digraph has an orientable embedding in which the given circuits are facial walks and there are exactly one or two other faces. 
This embedding then has maximum genus relative to the given circuits being facial walks.  When there is only one other face, it is necessarily bounded by an euler circuit. Consequently, if the numbers of vertices and edges have the same parity, a sufficiently dense digraph $D$ with a given directed euler circuit $T$ has an orientable embedding with exactly two faces, each bounded by an euler circuit, one of which is $T$. These results for digraphs give analogous results for graphs as immediate corollaries. The main theorem encompasses several special cases in the literature, such as when the digraph is a tournament.

\end{abstract}


\end{frontmatter}

% \linenumbers

\section{Introduction}

When does a graph or digraph have an orientable \emph{bi-eulerian} embedding?  This is, when can it be cellularly embedded in a orientable surface so that it has exactly two faces, each bounded by an euler circuit, such as shown in Figure \ref{BEembed}? Is it possible to specify one of the euler circuits in advance? 

\begin{figure}[h!]
     \centering
   \begin{subfigure}[l]{0.5\textwidth}
        \centering
    \includegraphics[clip, trim=0cm 8.5cm 0cm 0cm,scale=0.45]{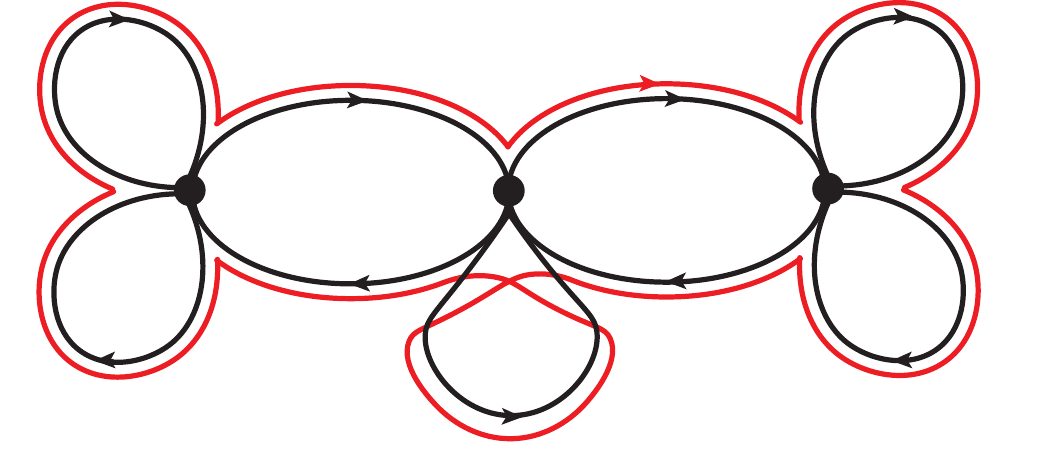}
        \caption{A digraph with a specified directed euler circuit.}
        \label{eulercirc}
     \end{subfigure}
        \hfill
        \begin{subfigure}[c]{0.7\textwidth}
        \centering
    \includegraphics[clip, trim=0cm 0cm 0cm 7.5cm,scale=0.6]{RibbonFinal.pdf}
        \caption{A directed embedding of the digraph with two faces.}
        \label{bieulemb}
     \end{subfigure}

\caption{A bi-eulerian orientable embedding of a digraph, with the specified euler circuit as one face (red), and one other euler circuit face (blue).  The embedding is represented as a ribbon graph, while sewing a  disc into each facial walk gives the surface.}
\label{BEembed}
\end{figure}

More generally, when is it possible to specify an arbitrary circuit decomposition of the edges or arcs and complete it to an embedding with just one more face, necessarily bounded by an euler circuit?  Finding such a face achieves a maximum genus orientable embedding  having the  circuits in a given decomposition as facial walks. This leads to the general question of determining the maximum genus of an embedding relative to a given circuit decomposition. Beyond topological graph theory, these questions arise in surprisingly diverse settings, including DNA self-assembly, Steiner triple systems, and latin squares.

Our main result, given in Theorem \ref{densemain}, answers these questions for sufficiently dense graphs and digraphs.  We prove that if every vertex of an $n$-vertex digraph has at least $(4n+2)/5$ neighbors, and a necessary parity condition holds, then it is indeed always possible to achieve these special embeddings of maximum orientable genus.  Our proof is constructive, so gives an algorithmic process for finding the desired embedding. Since the edges of an eulerian graph can always be directed to give an eulerian digraph, the stronger digraph results lead immediately to corollaries with analogous results for graphs. 

Such consequences include 
 Corollary \ref{undirdenseotef}, which says that if $G$ is a sufficiently dense eulerian graph with euler circuit $T$, then there is a $2$-face-colorable orientable embedding of $G$ with $T$ as the unique face of one color and with exactly one or two faces of the other color, depending on whether $|V(G)|+|E(G)|$ is even or odd, respectively.  This embedding has maximum genus among all $2$-face-colorable orientable embeddings of $G$.
When $|V(G)|+|E(G)|$ is even, then this is an orientable bi-eulerian embedding of $G$ with $T$ as a specified face, and the embedding has maximum genus among all orientable embeddings of $G$.  However, in general not every maximum genus embedding of such a $G$ is bi-eulerian.  

Our original motivation for studying bi-eulerian embeddings came from an applied problem in DNA self-assembly posed by Jonoska, Seeman and Wu \cite{JSW09}, which required a special closed walk in a graph for a DNA strand to follow.  The presence of such walks corresponds to the existence of an orientable \emph{edge-outer} embedding, namely an embedding with a special \emph{outer} face whose boundary uses every edge at least once.  The paper \cite{JSW09} proves the existence of orientable edge-outer embeddings, while \cite{EE-M19} provides a short algorithm to find them and shows that finding an optimal embedding (with shortest outer face) is NP-hard.  Neither \cite{JSW09} nor \cite{EE-M19}, however, provides any control over the number or sizes of the non-outer faces.

The startlingly simple (to state!) and intriguing questions in the first paragraph emerged from this application. Although determining the size of optimal edge-outer faces is hard in general, for eulerian graphs any optimal edge-outer face is necessarily bounded by an euler circuit. Thus, we seek to control the remaining faces in an optimal edge-outer embedding of an eulerian graph by specifying them in advance with a circuit decomposition. Of particular interest are bi-eulerian orientable embeddings, particularly when one of the circuits is specified in advance.

While our original motivation was DNA self-assembly, these and some closely related questions have also received considerable attention in various other special settings.  

In \cite{Edm65}, Edmonds proved that every eulerian graph has an bi-eulerian embedding in some surface, but noted that his main theorem was not sufficient to determine the orientability of the embedding.
A result of Kotzig \cite{Kot68} implies that in a $4$-regular graph every euler circuit, or more generally every circuit decomposition, can be extended to an embedding by adding an additional face bounded by an euler circuit; again the orientability is not determined. In \cite[Section 6]{EE-M2mod4} we generalized both of these results to characterize  when a circuit decomposition of an eulerian graph can be extended to a nonorientable embedding by adding an additional euler circuit face.

In this paper we focus on the more challenging orientable situation.  Conditions for the existence of an orientable bi-eulerian embedding where one euler circuit face is specified in advance were given by Bonnington, Conder, Morton, and McKenna \cite{BCMM02} for regular tournaments, and by the current authors \cite{EE-M2mod4}, based on the degrees of vertices modulo $4$. A series of papers, \cite{EGS20, GGS05, GMS20, GPS18} authored by Griggs and \siran, and sometimes also Erskine, Grannell, McCourt, or Psomas, discusses completing a triangular decomposition to an embedding by adding an euler circuit, for situations arising from a structures in design theory (Steiner triple systems, latin squares or symmetric $3$-configurations).

More details of these, and other, prior results are given in Section \ref{previousresults}, after reviewing necessary definitions and basic results in Section \ref{sec:term}.  We then state our main result and its various consequences and related results in Section \ref{sec:stateresults}, deferring the rather technical proof of the main result until Section \ref{MainProof}.  Section \ref{SupLemmas} contains the lemmas, some of independent interest, used for the proof of the main theorem.
We conclude in Section \ref{sec:conclusion} with some closing remarks.

An overview of some of the results in this paper has appeared in \cite{EE-Mdense}.

\section{Terminology} \label{sec:term}

We briefly recall some definitions and terminology from \cite{EE-M2mod4}.

\subsection{Representation of graphs and digraphs}

Our graphs and digraphs allow multiple edge or arcs, and loops, and we often need to specify a particular end of an edge or arc.  Therefore we define graphs and digraphs using half-edges and half-arcs, similarly to Fleischner~\cite{Fle90}.
 
A graph is a quadruple $G=(V,E^*,\incv,\mate)$ where $V$ (vertices) and $E^*$ (half-edges) are disjoint sets, $\incv: E^* \to V$ describes the incidence of each half-edge with a vertex, and $\mate: E^* \to E^*$ is a fixed-point-free involution that maps each half-edge to another half-edge.
An edge is an unordered pair $\{h, \mate(h)\}$ where $h \in E^*$, and we let $E$ denote the set of edges.  We use $E^*(v)$ to represent $\incv\iv(v)$, the set of half-edges incident with a given vertex $v$.  To specify a particular graph $G$ we write $V(G)$, $E^*(G)$, $\incv_G$, $\mate_G$, $E(G)$, and $E^*_G(v)$. 
The minimum degree of $G$ is denoted $\delta(G)$.

Similarly, a digraph is a quintuple $D=(V, A^+, A^-, \incv, \mate)$ where $V$ (vertices), $A^+$ (outgoing half-arcs) and $A^-$ (incoming half-arcs) are disjoint sets, and if $A^* = A^+ \cup A^-$ then
$\incv : A^* \to V$ describes incidences of half-arcs with vertices, and $\mate: A^* \to A^*$ is an involution that maps each element of $A^+$ to an element of $A^-$ and vice versa.
An arc of $D$ is an ordered pair $(g, \mate(g))$ where $g \in A^+$, and we let $A$ denote the set of arcs.
We use $A^+(v)$, $A^-(v)$, and $A^*(v)$ to represent the set of elements of $A^+$, $A^-$, and $A^*$, respectively, incident with a vertex $v$. To specify a particular digraph $D$ we write $V(D)$, $A^+(D)$, $A^-(D)$, $\incv_D$, $A^*(D)$, $\mate_D$, $A(D)$, $A^+_D(v)$, $A^-_D(v)$, and $A^*_D(v)$.
Two distinct vertices $u, v$ of $D$ are \emph{adjacent} if there is an arc (in either direction) with endvertices $u$ and $v$.

We warn the reader that we will use either the framework just described or standard definitions as convenient.
We follow the terminology of \cite{West} unless otherwise noted.

A graph, walk, digraph, or directed walk is \emph{nontrivial} if it has at least one edge or arc.

A \emph{walk} of \emph{length $\ell$} in a graph $G$ is a sequence $W = v_0 g_1 h_1 v_1 g_2 h_2 v_2 \dots\allowbreak v_{\ell-1} g_\ell h_\ell v_\ell$ where
$v_i \in V(G)$, $g_i, h_i \in E^*(G)$,
$\incv(g_i) = v_{i-1}$, $\incv(h_i) = v_i$ and $\mate(g_i) = h_i$ for all $i$ with $1 \le i \le \ell$.
A \emph{directed walk} in a digraph $D$ is similar to a walk, but with conditions $g_i \in A^+(D)$ and $h_i \in A^-(D)$ for all $i$ with $1 \le i \le \ell$.
Special types of walks include \emph{closed walks} ($v_0 = v_\ell$), \emph{paths} (no repeated vertices), \emph{cycles} (closed walks that use every edge of a $2$-regular subgraph exactly once), \emph{trails} (no repeated edges), and \emph{circuits} (closed trails). When it does not cause confusion, we do not distinguish between a closed (directed) walk and an equivalence class of closed (directed) walks under cyclic shifts, and we indicate such an equivalence class with parentheses, as $(v_0 g_1 h_1 v_1 \dots v_{\ell-1} g_\ell h_\ell)$.
If $W = v_0 g_1 \dots v_\ell$ and $W' = v'_0 g'_1 \dots v'_{\ell'}$ are two (directed) walks with $v_\ell = v'_0$, then the (directed) walk $v_0 g_1 \dots (v_\ell=v'_0) g'_1 \dots v'_{\ell'}$ obtained by concatenating $W$ and $W'$ is denoted by $W \mw W'$.

An \emph{euler circuit} in a graph is a circuit that contains all the edges and vertices.  A graph with an euler circuit is \emph{eulerian}, which is equivalent to being connected with every vertex of even degree.  Directed euler circuits and eulerian digraphs are defined similarly; an eulerian digraph is connected with every vertex having indegree equal to its outdegree.  Every eulerian graph has at least one \emph{circuit decomposition}, a collection of circuits using every edge exactly once.  Similarly, every eulerian digraph has at least one \emph{directed circuit decomposition}, using every arc exactly once.

\subsection{Embeddings and directed embeddings}

All graph embeddings in this paper are cellular.  We assume that the reader is familiar with cellular embeddings of graphs in surfaces and their various combinatorial representations; standard references are \cite{E-MM13, GT, MT}.
An embedding can represented by giving its collection of facial boundary walks, and we often do not distinguish between a face of an embedding and the corresponding facial walk.
An \emph{oriented} embedding is an embedding in an orientable surface with a specific global clockwise orientation.
Our embeddings will mostly be oriented and represented combinatorially by rotation systems specifying the clockwise order of half-edges or half-arcs around each vertex, and we use standard tracing procedures to determine the faces.
The \emph{Euler genus} $\gamma$ of a surface $\Sigma$ is the genus if $\Sigma$ is nonorientable and twice the genus if $\Sigma$ is orientable, and Euler's formula for a connected graph $G$ embedded in $\Sigma$ is $|V|-|E|+|F|=2-\gamma$.  

If we embed a digraph $D$ so that every face is bounded by a directed closed walk of $D$ (respecting the directions of the arcs in $D$), we call this a \emph{directed embedding} of $D$.  In an oriented directed embedding, each face is either a \emph{proface}, with facial walk directed clockwise, or an \emph{antiface}, with facial walk directed anticlockwise.  Directed embeddings have the following basic properties.

\begin{observation}[{Bonnington, Conder, Morton and McKenna \cite{BCMM02}}]\label{diremb-basic-s}\ 
\begin{enumerate}[(a), nosep]
\item\label{db-alt-s}
An embedding of a digraph is a directed embedding if and only if at each vertex the half-arcs, taken in rotational order, alternate between entering and leaving the vertex.

\item\label{db-eul-s}
Therefore, a connected digraph with a directed embedding must be eulerian, because at each vertex the numbers of entering and leaving half-arcs are equal.

\item\label{db-ori-s}
An eulerian digraph always has at least one oriented directed embedding, given by any rotation system in which the half-arcs alternate in direction around each vertex.

\item\label{db-ori-2fc}
An orientable directed embedding is always $2$-face-colorable, where (for a specific orientation of the surface) the profaces and the antifaces form the two color classes.

\end{enumerate}
\end{observation}

Part (d) above can be strengthened, as follows.

\begin{lemma}[{\cite[Lemma 2.2]{EE-M2mod4}}]
\label{diremb-ori-2fc}
A directed embedding is orientable if and only if it is $2$-face-colorable.
\end{lemma}

An embedding (directed embedding) of a graph (digraph) is \emph{bi-eulerian} if it has two faces, each bounded by an euler circuit (directed euler circuit).  Bi-eulerian embeddings of graphs may be orientable or nonorientable, but bi-eulerian directed embeddings of digraphs are always orientable, by Lemma \ref{diremb-ori-2fc}. 

The term \emph{bi-eulerian} or \emph{bieulerian} has occasionally been used in the literature, with a variety of meanings.  For example, Xuong \cite[p.~218]{Xuo79b} used it for embeddings with a single face covering every edge twice, Fleischner \cite[p.~VI.17]{Fle90} used it for special eulerian digraphs, and Chen and Fang \cite{ChFa22} used it for embeddings of bipartite eulerian graphs.  We trust that our reuse of this term will cause no confusion.

\subsection{Relative embeddings and upper relative embeddings}

Suppose $G$ is a connected graph and $\cW$ is a collection of closed walks in $G$.  If $\cW'$ is another collection of closed walks in $G$ such that $\cW$ and $\cW'$ together form the facial walks of an embedding $\Phi$, then we say $\Phi$ is an \emph{embedding of $G$ relative to $\cW$}.  We say the faces bounded by elements of $\cW$ and of $\cW'$ are \emph{specified} and \emph{new} faces, respectively.
If $\Phi$ is orientable, $\cW$ is a circuit decomposition $\cC$ of $G$, and there are only one or two new faces, then we say that $\Phi$ is an \emph{upper embedding of $G$ relative to $\cC$}.  (We require $\cW$ to be a circuit decomposition to avoid complications that arise in more general situations.)  In particular, if we complete $\cC$ to an orientable embedding by adding an euler circuit, then we have an upper relative embedding.
If an upper embedding of $G$ relative to $\cC$ exists, it has maximum genus over all orientable embeddings of $G$ relative to $\cC$.

We can also apply these concepts to collections of directed walks, directed circuit decompositions, and directed embeddings of digraphs.  The following observation relies on Lemma \ref{diremb-ori-2fc} for (a) and on evenness of the Euler genus of an orientable embedding for (b).  Part (b) tells us whether an upper relative embedding will have just one new face or two new faces.

\begin{observation}\label{diremb-relcd}
Suppose $\cC$ is a directed circuit decomposition of an eulerian digraph $D$ and $\Phi$ is a directed embedding of $D$ relative to $\cC$.
\begin{enumerate}[(a), nosep]
\item
Then $\Phi$ is orientable, and the surface can be oriented so that $\cC$ is the collection of profaces.
\item
The number of new faces of $\Phi$ has the same parity as $|V(D)|+|A(D)|+|\cC|$.
\end{enumerate}
\end{observation}

We will use the following result on the existence of oriented directed embeddings relative to a directed circuit decomposition.

\begin{lemma}\label{dcdembedding}
Suppose $\cC$ is a directed circuit decomposition of an eulerian digraph $D$.  Then there exists an oriented directed embedding of $D$ in which the profaces are precisely the elements of $\cC$.
\end{lemma}

\begin{proof}
Since $\cC$ is a directed circuit decomposition of $D$, every half-arc occurs exactly once in $\cC$.  For each $h \in A^-_D$, let $\fw(h) \in A^+_D$ be the half-arc that follows $h$ in some element of $\cC$.  Then $\fw(A^-(v)) = A^+(v)$ for all $v \in V(D)$.

We construct a rotation at each vertex $v$.  Suppose $v$ has $\indeg(v) = \outdeg(v) = d$, and $A^-(v) = \{h_0, h_1, \ldots, h_{d-1}\}$.  Let $g_i = \fw(h_i)$ for $0 \le i \le d-1$, and define the clockwise rotation at $v$ to be $(g_0 h_0 g_1 h_1 \ldots g_{d-1} h_{d-1})$.  Since the half-arcs alternate in direction at each $v$, these rotations specify an orientable directed embedding of $D$.  Each proface that enters $v$ on $h_i$ leaves on $g_i=\fw(h_i)$, so the profaces join the arcs together in the same way that the circuits in $\cC$ do.  Thus, the profaces are exactly the elements of $\cC$.
\end{proof}

This lemma was used implicitly by Bonnington et al.~\cite[p.~13]{BCMM02} in the case where $\cC$ consists of a directed euler circuit.

\section{Previous results}\label{previousresults}

Graph embeddings with euler circuit faces have been a subject of interest for nearly sixty years.  In this section we discuss relevant prior results.  We focus on orientable embeddings relative to a circuit decomposition, but briefly discuss more general results first.

Edmonds \cite[p.~123]{Edm65} showed that every eulerian graph has a bi-eulerian embedding.  A result of Kotzig \cite[Theorem 3]{Kot68} can be interpreted as saying that every circuit decomposition of a $4$-regular graph can be extended to an embedding by adding an euler circuit.  In both Edmonds' and Kotzig's results nothing is specified about the orientability of the resulting embedding.  We showed \cite[Theorem 6.5]{EE-M2mod4} that a circuit decomposition of an arbitrary eulerian graph can be extended to an embedding by adding an euler circuit, and the embedding can be guaranteed to be nonorientable unless every block of the graph is a cycle and the circuit decomposition is the collection of cycles.  
This extends some special cases in the literature, such as the nonorientable parts of Theorems \ref{steiner} and \ref{latin} below.
In particular, every eulerian graph that is not a cycle has a nonorientable bi-eulerian embedding where one of the euler circuit faces can be specified in advance.

We turn now to orientable embeddings.  If we have a circuit decomposition $\cC$ of an eulerian graph, then we can find an upper embedding relative to $\cC$ if we can complete $\cC$ to an orientable embedding using just one face (which is necessarily an euler circuit) or two faces.  Similarly, if we have a directed circuit decomposition $\cC$ of an eulerian digraph, then we can find an upper directed embedding relative to $\cC$ if we can complete $\cC$ to a directed embedding using just one face (which must be a directed euler circuit) or two faces.  By Observation \ref{diremb-relcd} the embedding is necessarily orientable, and we can assume that $\cC$ forms the profaces and the added face or faces form the antifaces.

General formulas for relative oriented maximum genus were given by Bonnington \cite{Bon94} and Archdeacon, Bonnington, and \siran. \cite{ABS98}, which generalize maximum orientable genus formulas due to Xuong \cite{Xuo79b} and \nebesky. \cite{Neb81}, respectively.  For a relative oriented embedding of a graph, each specified face has a specified direction that should be clockwise in the surface.  If we have a directed circuit decomposition $\cC$ of an eulerian digraph $D$, we can consider this as a collection of oriented circuits in the underlying graph $G$, and the formulas from \cite{ABS98, Bon94} can be applied to find the maximum genus of a (necessarily orientable) directed embedding relative to $\cC$.  However, it seems to be difficult to use these formulas to derive easily applicable conditions for the existence of a relative upper embedding (although we do know of one situation where we can do this, as we discuss at the end of this section).

The first result that we are aware of giving a specific condition for finding a relative upper embedding of some type of circuit decomposition is the following.

\begin{theorem}[Bonnington, Conder, Morton, and McKenna, {\cite{BCMM02}}]
\label{tournament}
Let $T$ be a directed euler circuit in a regular (equivalently, eulerian) tournament $D$. Then there is an orientable directed embedding of $D$ with $T$ as the only proface and with at most two antifaces, i.e., an upper directed embedding of $D$ relative to $\{T\}$.
\end{theorem}

There is a series of papers \cite{EGS20, GGS05, GMS20, GPS18} involving Griggs and \siran. as authors, and also Erskine, Grannell, McCourt, and Psomas, which discuss upper embeddings relative to a triangular decomposition of a graph or digraph, and more specifically completing such a decomposition to an embedding by adding an euler circuit.  They are interested in triangular decompositions that arise from structures in design theory such as Steiner triple systems, symmetric $3$-configurations, and latin squares.  Their first results were for undirected situations.

\begin{theorem}[Granell, Griggs and \siran. {\cite{GGS05}}]
\label{steiner}
Let $\cC$ be a Steiner triple system of order $n \ge 7$, i.e., a decomposition of the complete graph $G=K_n$ into triangles.
Then there are both orientable and nonorientable embeddings of $G$ with the elements of $\cC$ as faces and exactly one additional euler circuit face.
\end{theorem}

\begin{theorem}[Griggs, Psomas, and \siran. {\cite{GPS18}}]
\label{latin}
Let $\cC$ be a latin square of order $p$, i.e., a decomposition of the balanced complete tripartite graph $G=K_{p,p,p}$ into triangles.
If $p \ge 2$ then there is a nonorientable embedding of $G$ with the elements of $\cC$ as faces and with exactly one additional euler circuit face.
If $p$ is odd then there is an orientable embedding of $G$ with the elements of $\cC$ as faces and with exactly one additional euler circuit face.
\end{theorem}

Griggs, McCourt and \siran. strengthened Theorems \ref{steiner} and \ref{latin} by adding specified directions to the triangle decompositions.  This effectively changes the setting from undirected graphs to digraphs, and their results can be stated as follows.

\begin{theorem}[Griggs, McCourt, and \siran. {\cite[Theorem 1.1]{GMS20}}]
\label{dirsteiner}
Let $\cC$ be an oriented Steiner triple system, i.e., a decomposition of a regular tournament $D$ into directed triangles.  Then there is an orientable directed embedding of $D$ with the elements of $\cC$ as the profaces and with exactly one directed euler circuit antiface.
\end{theorem}

\begin{theorem}[Griggs, McCourt, and \siran. {\cite[Theorem 1.2]{GMS20}}]
\label{dirlatin}
Let $\cC$ be an oriented latin square of odd order, i.e., a decomposition of an eulerian orientation $D$ of a balanced complete tripartite graph $K_{p,p,p}$ into directed triangles.  Then there is an orientable directed embedding of $D$ with the elements of $\cC$ as the profaces and with exactly one directed euler circuit antiface.
\end{theorem}

Erskine, Griggs, and \siran. also investigated the situation for graphs and digraphs obtained from symmetric $3$-configurations, obtaining a positive result for small cases but a negative result in general.

\begin{theorem}[Erskine, Griggs, and \siran. {\cite[Theorems 2.1 and 3.2]{EGS20}}]
\label{configuration}
Suppose $G$ is the associated graph of a symmetric configuration $n_3$, i.e., $G$ is a $6$-regular $n$-vertex simple graph with a decomposition $\cT$ into triangles.  Let $\cC$ be a collection of directed triangles obtained by directing the elements of $\cT$, and let $D$ be the digraph obtained from $G$ by directing the edges as specified by $\cC$.

If $n$ is odd and $7 \le n \le 19$, then every for every such $D$ and $\cC$ there is an orientable directed embedding of $D$ with the elements of $\cC$ as the profaces and with exactly one directed euler circuit antiface.

If $n$ is odd and $n \ge 21$ then there exists such a graph $G=G_0$ with decomposition $\cT=\cT_0$ such that there is no orientable embedding of $G_0$ with the elements of $\cT_0$ as faces and exactly one additional face.  Thus, for every $D$ and $\cC$ derived as above from $G_0$ and $\cT_0$ there is no orientable directed embedding of $D$ with the elements of $\cC$ as the profaces and with exactly one antiface.
\end{theorem}

The proofs of Theorems \ref{steiner} through \ref{configuration} use arguments that are specific to triangle decompositions.

The current authors investigated maximum genus of embeddings relative to an euler circuit, and of directed embeddings relative to a directed euler circuit, in \cite{EE-M2mod4}.  We showed \cite{EE-M2mod4} that in every eulerian graph or digraph where every vertex has degree $2$ mod $4$, every (directed) euler circuit can be extended to an orientable (directed) embedding by adding a second (directed) euler circuit, giving a relative upper embedding.  One way to derive this result is by applying Bonnington's formula for maximum genus of a relative oriented embedding \cite{Bon94}, although that is not the approach used in \cite{EE-M2mod4}.  Theorem \ref{configuration} above from \cite{EGS20} shows that our result unfortunately does not generalize to extending an arbitrary circuit decomposition; see \cite[Section 5]{EE-M2mod4} for further discussion.

The main results of this paper (Theorem \ref{densemain} and its corollaries) encompass a number of the results mentioned above, specifically Theorem \ref{tournament}, the orientable part of Theorem \ref{steiner}, and Theorem \ref{dirsteiner}.  Our main results are not, however, strong enough to imply the orientable part of Theorem \ref{latin}, or Theorem \ref{dirlatin}, because balanced complete tripartite graphs do not satisfy our density condition.

\section{Statement of results}\label{sec:stateresults}

In this section we state our results for `dense' eulerian graphs and digraphs.  We also provide results for all eulerian graphs and digraphs with only one or two vertices (but arbitrarily many edges or arcs).

\subsection{The main result and some consequences}
\label{MainAndConseq}

In this subsection we consider `dense' eulerian digraphs, where `dense' means that each vertex is adjacent to at least roughly $4/5$ of the other vertices.
We show that given a collection $\cC$ of edge-disjoint directed circuits in a dense eulerian digraph, we can find a directed orientable embedding in which all elements of $\cC$ are profaces, and there are at most two antifaces.
This means that we can determine the maximum genus of an orientable directed embedding in which the elements of $\cC$ are required to be profaces.
As a special case, and subject to a necessary parity condition we can find an orientable bi-eulerian directed embedding where one of the euler circuit faces may be specified in advance.

Our main result, and many of our arguments, rely only on information about whether two vertices are adjacent in a digraph, and not how many arcs are between them or directions of the arcs.  Thus, it is often convenient to work with the undirected simple graph underlying a digraph $D$, which we denote by $\usg{D}$.  The  graph $\usg{D}$ has the same vertex set as $D$,  and distinct vertices $u$ and $v$ are adjacent in $\usg{D}$  if and only if they are adjacent in $D$.
Moreover, if $W$ is a directed walk (such as a facial walk of a directed embedding) in $D$, then $\usg{W}$ is the corresponding walk in $\usg{D}$: we take the sequence of vertices in $W$, eliminating consecutive repetitions of a vertex (created by following loops).  We also treat $\usg{W}$ as a subgraph of $\usg{D}$; whether we are thinking of it as a walk or a subgraph should be clear from context.

The main result of this paper is the following.

\def\densemaintext{% Text is defined as a macro so we can restate later 
Let $D$ be an $n$-vertex eulerian digraph where $\delta(\usg{D})$, the minimum degree of the underlying simple undirected graph of $D$, is at least $(4n+2)/5$.
Let $\cC$ be a directed circuit decomposition of $D$.
Then there is a directed embedding of $D$ in an oriented surface (i.e., an orientable surface with a specific orientation) with the elements of $\cC$ as the profaces and with exactly one or two antifaces, depending on whether $|V(D)|+|A(D)|+|\cC|$ is odd or even, respectively.

This embedding has maximum genus among all orientable directed embeddings of $D$ in which all elements of $\cC$ are faces.

}

\begin{theorem}\label{densemain}
\densemaintext
\end{theorem}

As mentioned in Section \ref{previousresults}, Theorem \ref{densemain} generalizes Theorems \ref{tournament} and \ref{dirsteiner}, but not Theorem \ref{dirlatin}, because $K_{p,p,p}$ is not dense enough.
Because the proof of Theorem \ref{densemain} is rather long and technical, we defer it to Sections \ref{SupLemmas} and \ref{MainProof}, in order to state here some immediate corollaries that illustrate its impact and application.

If $\cC$ consists of a single directed euler circuit, then we have the following.

\begin{corollary}\label{denseotef}
Let $D$ be an $n$-vertex eulerian digraph where $\delta(\usg{D}) \ge (4n+2)/5$.
Let $T$ be a directed euler circuit in $D$.  Then there is an oriented directed embedding of $D$ with $T$ as the only proface, and with exactly one or two antifaces, depending on whether $|V(D)|+|A(D)|$ is even or odd, respectively.

This embedding has maximum genus among all orientable directed embeddings of $D$.  When $|V(D)|+|A(D)|$ is even this embedding is an orientable bi-eulerian directed embedding of $D$ with $T$ as a specified face.
\end{corollary}

When $|V(D)|+|A(D)|$ is even in Corollary \ref{denseotef}, we can also apply Observation \ref{diremb-basic-s}\ref{db-ori-2fc} to conclude that \emph{every} maximum genus orientable directed embedding of $D$ is bi-eulerian.

By extending a set $\cC$ of arc-disjoint directed circuits to a circuit decomposition using an euler circuit in each component of the subgraph induced by the unused arcs, we also obtain the following corollary.

\begin{corollary}\label{densemaincor}
Let $D$ be an $n$-vertex eulerian digraph where $\delta(\usg{D}) \ge (4n+2)/5$.
Let $\cC$ be a set of arc-disjoint directed circuits in $D$, let $A(\cC)$ be the set of arcs used by elements of $\cC$, and let $\alpha(\cC)$ be the number of components of $D - A(\cC)$ with at least one edge.
Then there is an oriented directed embedding of $D$ with $|\cC| + \alpha(\cC)$ profaces and with exactly one or two antifaces, depending on whether $|V(D)|+|A(D)|+|\cC|+\alpha(\cC)$ is odd or even, respectively.

This embedding has maximum genus among all orientable directed embeddings of $D$ in which all elements of $\cC$ are faces.
\end{corollary}

We also have corollaries for undirected graphs.  Let $\sg{G}$ denote the underlying simple graph of a graph $G$.  We state the undirected counterparts of Theorem \ref{densemain} and Corollary \ref{denseotef}.  In each case the first paragraph is obtained by applying the digraph result to an eulerian orientation of the graph $G$ that makes each element of $\cC$ into a directed circuit, and the conclusions about maximum genus in the second paragraph are easily verified.
There is also an undirected version of Corollary \ref{densemaincor}, whose statement we leave to the reader.

\begin{corollary}\label{undirdensemain}
Let $G$ be an $n$-vertex eulerian graph where $\delta(\sg{G})$, the minimum degree of the underlying simple graph of $G$, is at least $(4n+2)/5$.
Let $\cC$ be a circuit decomposition of $G$.
Then there is a $2$-face-colorable orientable embedding of $G$ with the elements of $\cC$ as one of the color classes and with exactly one or two faces of the other color, depending on whether $|V(G)|+|E(G)|+|\cC|$ is odd or even, respectively.

This embedding has maximum genus among all orientable embeddings of $G$ in which all elements of $\cC$ are faces.
\end{corollary}

\begin{corollary}\label{undirdenseotef}
Let $G$ be an $n$-vertex eulerian graph where $\delta(\sg{G}) \ge (4n+2)/5$.
Let $T$ be an euler circuit in $G$.  Then there is a $2$-face-colorable orientable embedding of $G$ with $T$ as the unique face of one color and with exactly one or two faces of the other color, depending on whether $|V(G)|+|E(G)|$ is even or odd, respectively.

This embedding has maximum genus among all $2$-face-colorable orientable embeddings of $G$.
When $|V(G)|+|E(G)|$ is even this embedding is an orientable bi-eulerian embedding of $G$ with $T$ as a specified face, and the embedding has maximum genus among all orientable embeddings of $G$.
\end{corollary}

When $|V(D)|+|A(D)|$ is even in Corollary \ref{denseotef}, all maximum genus directed embeddings are bi-eulerian.  However, when $|V(G)|+|E(G)|$ is even in Corollary \ref{undirdenseotef}, the embeddings provided by the corollary are bi-eulerian maximum genus embeddings, but there may be other maximum genus embeddings that are not bi-eulerian, as the following example shows.

\begin{example}
Suppose we have an eulerian graph $G$ satisfying our density condition, with a decomposition into three circuits $C_1, C_2, C_3$, all of which contain the same vertex $v$.  (Such situations are not hard to construct.)  Apply Corollary \ref{undirdensemain} to complete $\{C_1, C_2, C_3\}$ to an embedding $\Phi$ using an euler circuit $T$.  
By applying an undirected version of part (a) of the Three Face Lemma (Lemma \ref{3face} below) to $\Phi$ at $v$
we can merge $C_1$, $C_2$, and $C_3$ into a single face, giving a bi-eulerian embedding, or alternatively merge $C_1$, $C_2$ and $T$ into a single face, giving a two-face orientable embedding that is not bi-eulerian.
\end{example}

Our results can be combined with known results on the existence of triangular decompositions in dense graphs, to give maximum genus embeddings using these decompositions.  We say a graph is \emph{triangle-divisible} if its number of edges is divisible by $3$ and every vertex has even degree.  In \cite{NW70}, Nash-Williams conjectured that there is some constant $c$ such that for sufficiently large $n$, every $n$-vertex triangle-divisible simple graph $G$ with $\delta(G) \ge cn$ has a decomposition into edge-disjoint triangles, and Graham pointed out that $c$ must be at least $3/4$.  The current best result on Nash-Williams' conjecture is as follows.

\begin{theorem}[{Delcourt and Postle \cite{DePo21}}]\label{nwbest}
For every $\ep > 0$ there is $N(\ep)$ such that every $n$-vertex triangle-divisible simple graph $G$ with $n \ge N(\ep)$ and $\delta(G) \ge ((7+\sqrt{21})/14 + \ep) n$ has a decomposition into edge-disjoint triangles.
\end{theorem}

The value of $(7+\sqrt{21})/14$ is about $0.827327 > 4/5$, and triangle-divisible simple graphs $G$ with $\delta(G) \ge n/2$ are eulerian, so applying Corollary \ref{undirdensemain} gives the following.

\begin{corollary}\label{trieuler}
For every $\ep > 0$ there is $N'(\ep)$ with the following properties.  Every $n$-vertex triangle-divisible simple graph $G$ with $n \ge N'(\ep)$ and $\delta(G) \ge ((7+\sqrt{21})/14 + \ep) n$ has a triangular decomposition $\cC$.  For every such decomposition $\cC$ there is a $2$-face-colorable orientable embedding where all faces of one color are bounded by the elements of $\cC$ and there are exactly one or two faces of the other color, depending on whether $|V(G)|$ is odd or even, respectively. This embedding has maximum genus among all orientable embeddings in which all elements of $\cC$ are faces.
\end{corollary}

This may be regarded as a strengthening of the orientable part of Theorem \ref{steiner}, since a complete graph has a triangular decomposition, corresponding to a Steiner triple system, if and only if it is triangle-divisible.

Corollary \ref{trieuler} also tells us that a sufficiently large sufficiently dense triangle-divisible simple graph $G$ always has an edge-outer embedding that is optimal in three respects: it has an outer face of minimum possible length (the euler circuit face), the maximum possible number of faces other than the outer face (the triangular faces), and hence it attains the minimum genus numerically possible for an edge-outer embedding of a simple graph with $|V(G)|$ vertices and $|E(G)|$ edges.

Future strengthenings of Theorem \ref{nwbest} may also provide improvements to Corollary \ref{trieuler}.

\subsection{Graphs and digraphs with one or two vertices}

Since $\usg{D}$ is simple, the degree condition in Theorem \ref{densemain} gives $n-1 \ge \delta(\usg{D}) \ge (4n+2)/5$, which requires that $n \ge 7$.
However, our techniques can also give results covering all eulerian graphs and digraphs with just one or two vertices (but arbitrarily many edges or arcs).

An \emph{edge-cut} $S$ in a digraph $D$ is the set of all arcs with one end $U$ and other end in $\overline U = V(D)-U$ for some proper nonempty subset $U$ of $V(D)$, and a \emph{$k$-edge-cut} is an edge-cut of cardinality $k$.  In an eulerian digraph $D$ every edge-cut has even cardinality.  Recall that for an integer $a$, $a \bmod 2$ is equal to $0$ if $a$ is even and $1$ if $a$ is odd.

\begin{proposition}\label{small}
Let $D$ be an eulerian digraph with $|V(D)| \le 2$ and let $\cC$ be a directed circuit decomposition of $D$.
\begin{enumerate}[(a), nosep]
\item
If $|V(D)|=1$, or if $|V(D)|=2$ and $D$ has no $2$-edge-cut, then there is an oriented directed embedding of $D$ with the elements of $\cC$ as the profaces and with exactly one or two antifaces, depending on whether $|V(D)|+|A(D)|+|\cC|$ is odd or even, respectively.

\item
Suppose that $|V(D)|=2$ and $D$ has a $2$-edge-cut.
Let $V(D)=\{v_1,v_2\}$, for each $i \in \{1,2\}$ let $\al_i$ be the number of directed loops incident with $v_i$ in $D$, and let $\ga_i$ be the number of elements of $\cC$ that are incident only with $v_i$, not with $v_{3-i}$.  Then there is an oriented directed embedding of $D$ with the elements of $\cC$ as the profaces and with exactly $((\al_1+\ga_1) \bmod 2) + ((\al_2+\ga_2) \bmod 2) + 1$ antifaces.
\end{enumerate}
In both cases the embedding described has maximum genus among all orientable directed embeddings of $D$ in which all elements of $\cC$ are faces.
\end{proposition}

Proposition \ref{small} is proved below, after Lemma \ref{interlacing}.  Immediate corollaries of Proposition \ref{small} can also be stated for bi-eulerian embeddings, partial circuit decompositions (collections of arc-disjoint circuits), and undirected graphs.  We leave the formulation of these to the reader.

\section{Supporting lemmas}\label{SupLemmas}

The proof of our main theorem, Theorem \ref{densemain}, depends on the lemmas of this section.  These lemmas facilitate combining antifaces in an oriented directed embedding without changing the profaces. While we developed these lemmas to prove Theorem \ref{densemain}, they are also of generic utility for manipulating embedded (di)graphs.

Our proof of Theorem \ref{densemain} was inspired by the approach used to prove Theorem \ref{tournament} in Bonnington et al.~\cite[Section 3]{BCMM02} .  However, since we treat general dense graphs instead of the special case of tournaments, we necessarily develop a number of new ideas. Our Three Face Lemma (\ref{3face}) and Interlaced Faces Lemma (\ref{interlacing}) simplify and generalize some arguments from \cite[Lemma 3.1 and proof of Lemma 3.4]{BCMM02}.
Our Three Neighbor Lemma (\ref{3adj}) generalizes \cite[Lemma 3.3]{BCMM02} and our Diamond Lemma (\ref{diamond}) summarizes and generalizes some reasoning from \cite[proof of Theorem 3.1]{BCMM02}.

Since we address digraphs here that are not orientations of complete graphs or even simple graphs, we need new sufficient conditions for when the results above can be applied. We provide these in the Three Neighbor Corollary (\ref{3adjcor}), Big and Moderate Faces Corollary (\ref{BMFlemma}), and Diamond Corollary (\ref{diamondcor}).  We also introduce a number of new ideas in the Bipartite Degeneracy Lemma (\ref{bipdegen}), Division Lemma (\ref{division}), Blow Up Lemma (\ref{blowup}), and how we use touch graphs in Section \ref{MainProof} for the proof of the main result.  Our more general setting also requires a significantly more complicated analysis.  Moreover, since our approach only manipulates antifaces, and does not depend on the profaces, it can be applied to situations where the profaces form an arbitrary directed circuit decomposition, not just a directed euler circuit.

\subsection{General lemmas for merging antifaces}

We begin with some general results.
The Three Face Lemma, Lemma \ref{3face} below, is our main tool for merging or rearranging faces.  We follow it with a useful consequence, Lemma \ref{interlacing}.

The Three Face Lemma (\ref{3face}) describes what happens if we modify the rotation at a vertex $v$ by breaking it in three places and reassembling it.  The places where we break it correspond to faces $A$, $B$, and $C$.
The lemma covers the cases where the three faces are distinct, or where $B$ is distinct from $A=C$.
(We can also investigate what happens when $A=B=C$, although we do not need this for our results.  There are two situations: one is the inverse of part (a), where one antiface splits into three antifaces, and in the other situation the antiface has its boundary rearranged but remains a single antiface.)

\begin{lemma}[Three Face Lemma]\label{3face}
Let $\Phi$ be an oriented directed embedding of an eulerian digraph $D$, and $v \in V(D)$.

\smallskip
(a) Suppose $A$, $B$, and $C$ are distinct antifaces that each contain $v$.  Then there is an oriented directed embedding $\Phi'$ of $D$ that has the same profaces and antifaces as $\Phi$ except that $A$, $B$ and $C$ are merged into a single antiface.

\smallskip
(b) Suppose $A$ and $B$ are distinct antifaces that each contain $v$.  Suppose further that $A$ can be written as $A = A_1 \mw A_2$ where $A_1$ and $A_2$ are nontrivial $vv$-walks.  Also consider $B$ as a $vv$-walk.  Then there is an oriented directed embedding $\Phi'$ of $D$ that has the same profaces and antifaces as $\Phi$, except that $A$ and $B$ are replaced by two new antifaces $A'= A_i \mw B$ and $B' = A_{3-i}$ for some $i \in \{1,2\}$.
\end{lemma}

\begin{figure}[ht!]
  \centering
  \begin{subfigure}{\textwidth}
    \hbox to\hsize{\hfil\includegraphics[scale=1.2]{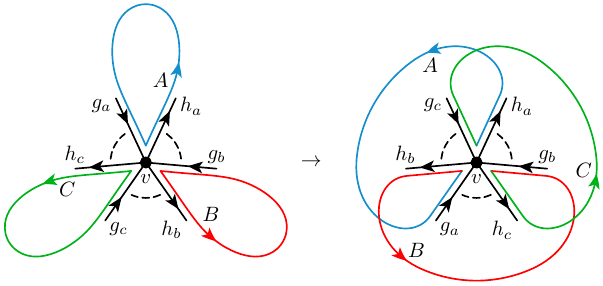}\hfil}
    \subcaption{Three distinct initial antifaces.}\label{i1a}
  \end{subfigure}
  \bigskip
  \begin{subfigure}{\textwidth}
    \hbox to \hsize{\hfil\includegraphics[scale=1.2]{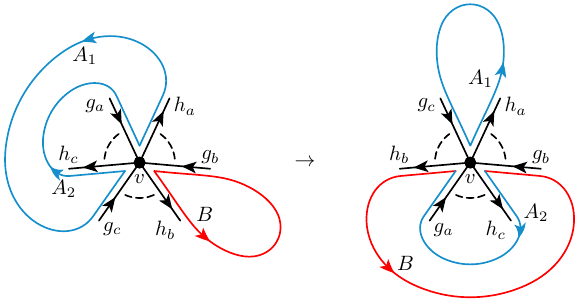}\hfil}
    \subcaption{Two initial antifaces.}\label{i1b}
  \end{subfigure}
  \caption{Cases for the Three Face Lemma.}
\label{fig-3face}
\end{figure}

\begin{proof}
Suppose the clockwise rotation at $v$ in $\Phi$ is $(g_0 h_0 g_1 h_1 \dots g_{d-1} h_{d-1})$ where $A^-(v) = \{g_i \;|\; i \in \mZ_d\}$ and $A^+(v) = \{h_i \;|\; i \in \mZ_d\}$.
Given $a, b, c$ in (increasing) cyclic order in $\mZ_d = \{0, 1, \dots, d-1\}$, let $\Phi'$ be the oriented directed embedding obtained from $\Phi$ by changing the rotation at $v$ from $$(
	h_a g_{a+1} h_{a+1} \dots g_b \;\;
	h_b g_{b+1} h_{b+1} \dots g_c \;\;
	h_c g_{c+1} h_{c+1} \dots g_a  )$$
to $$(
	h_a g_{a+1} h_{a+1} \dots g_b \;\;
	h_c g_{c+1} h_{c+1} \dots g_a \;\;
	h_b g_{b+1} h_{b+1} \dots g_c ).$$

\smallskip
(a) We may assume that our three distinct faces are represented as $vv$-walks with the first outgoing half-arc at $v$ of $A$ denoted by $h_a$, of $B$ by $h_b$, and of $C$ by $h_c$.  Thus, $A = v h_a \dots g_a v$, $B = v h_b \dots g_b v$, and $C = v h_c \dots g_c v$. Construct $\Phi'$ as above. Faces other than $A, B, C$ are unchanged, and by tracing the antiface starting $v h_a \dots$ we see that $A, B, C$ are merged into a single antiface $A \mw B \mw C$.
See Figure \ref{fig-3face}\subref{i1a}.

\smallskip
(b) We may assume (by swapping $A_1$ and $A_2$ if necessary)
that $A_1 = v h_a \dots g_c v$, $A_2 = v h_c \dots g_a v$, and
$B = v h_b \dots g_b v$.
Construct $\Phi'$ as above. Faces other than $A$ and $B$ are unchanged, and by tracing faces we see that $A$ and $B$ are replaced by new antifaces
$v h_c \dots g_a v h_b \dots g_b v = A_2 \mw B$ and $v h_a \dots g_c v = A_1$.
See Figure \ref{fig-3face}\subref{i1b}.
\end{proof}

 In part (b) of the statement of Lemma \ref{3face}, $A_i$ is whichever of $A_1$ or $A_2$ the proof relabels as $A_2$, which is the one that has its initial and final half-arcs further from $g_b$ and $h_b$ in the rotation around $v$.  In general we cannot choose the value of $i$.

There is an undirected version of Lemma \ref{3face}(a), which allows us to combine three distinct faces incident with the same vertex into a single face by adding a handle to an orientable embedding of a graph.
Ellingham and Weaver \cite{EW08} explain how this idea can be regarded as a special case of an operation that merges two pairs of faces incident with a given vertex.
Similar ideas and special cases involving local modification of an embedding around a single vertex or single face have been used for a long time, for example by Ringel \cite[p.~120]{Ri61} in one of the papers proving the Map Color Theorem, by Duke \cite[Theorem 3.2]{Duk66} in showing that the orientable genus range of a graph is an interval, and by Xuong \cite[Figure 4]{Xuo79b} in finding maximum genus embeddings.

The operation in Lemma \ref{3face} does not change the profaces of the embedding at all.  In our arguments the profaces will be specified in advance and never modified.
However, if desired, by switching the roles of profaces and antifaces, we could also apply Lemma \ref{3face} to merge or rearrange profaces without changing antifaces.

We say distinct vertices $x, y$ are \emph{interlaced} on a closed walk (or face) $W$ if (possibly after cyclic shifting) $W$ can be written as $x \dots y \dots x \dots y \dots$. 
Many of the following lemmas give conditions that assure the existence of, or allow us to create, at least one pair of interlaced vertices.  We eventually use interlaced pairs of vertices to merge antifaces, thus helping us reach the ultimate goal of only one or two antifaces.

\begin{lemma}[Interlaced Faces Lemma]
\label{interlacing}
Let $\Phi$ be an oriented directed embedding of an eulerian digraph $D$, and $x, y \in V(D)$.  Suppose that there are three distinct antifaces $A, B, C$ such that
$x$ and $y$ are interlaced on $A$,
while $x$ occurs on $B$ and $y$ occurs on $C$.
Then there is an oriented directed embedding $\Phi'$ of $D$ that has the same profaces and antifaces as $\Phi$ except that $A$, $B$ and $C$ are replaced by a single antiface.
\end{lemma}

\begin{proof}
Write $A = A_1 \mw A_2$ where $A_1$ and $A_2$ are nontrivial walks of the form $x \dots y \dots x$.  By Lemma \ref{3face}(b) we may modify $\Phi$ so as to replace $A$ and $B$ by two antifaces $A_i \mw B$ and $A_{3-i}$.  Now $A_i \mw B$, $A_{3-i}$, and $C$ are three distinct antifaces at $y$, so by Lemma \ref{3face}(a) we may combine them into a single antiface.
\end{proof}

Now we consider the situation of Theorem \ref{densemain} or Proposition \ref{small}.  We suppose $D$ is an $n$-vertex eulerian digraph, and $\cC$ is a directed circuit decomposition of $D$.
By Lemma \ref{dcdembedding} there exists an oriented directed embedding $\Phi$ of $D$ in which the profaces are exactly the elements of $\cC$.
Let $\cA$ be the set of antifaces of $\Phi$.
Our overall goal now is to show that we can decrease $|\cA|$ to $1$ or $2$ without changing the profaces.
If there is a vertex that belongs to three antifaces, we will apply Lemma \ref{3face}(a) to reduce the number of antifaces.  Otherwise, each vertex belongs to at most two antifaces, and we say the embedding is \emph{locally irreducible}.
In our arguments the profaces of the embedding never change.

At this point we have developed the tools necessary to prove Proposition \ref{small}, which covers graphs with one or two vertices.

\begin{proof}[Proof of Proposition \ref{small}]
(a) By Lemma \ref{dcdembedding} there is an oriented directed embedding $\Phi_0$ of $D$ with $\cC$ as the set of profaces.  By repeatedly applying the Three Face Lemma (\ref{3face}) beginning with $\Phi_0$, we can obtain a locally irreducible embedding $\Phi_1$ with $\cC$ as the set of profaces.  We show that we can obtain $\Phi_2$ with at most two antifaces and with $\cC$ as the set of profaces.

If $\Phi_1$ has at most two antifaces, which holds if $|V(D)|=1$, then we can take $\Phi_2 = \Phi_1$.  So we may suppose that $V(D) = \{v_1, v_2\}$ and $\Phi_1$ has at least three antifaces.  Some antiface $A$ must appear at both $v_1$ and $v_2$.  Let $B$ and $C$ be two other antifaces.
If there is a fourth antiface, or if either of $B$ or $C$ is incident with both $v_1$ and $v_2$, or if $B$ and $C$ are incident with the same vertex $v_i$, then $\Phi_1$ is not locally irreducible.  So there are only three antifaces $A, B, C$, and we may assume without loss of generality that $B$ is only incident with $v_1$ and $C$ is only incident with $v_2$.  Thus, $A$ uses all arcs between $v_1$ and $v_2$, and since there are at least four such arcs, $v_1$ and $v_2$ are interlaced on $A$.  We can therefore apply the Interlaced Faces Lemma (\ref{interlacing}) to obtain $\Phi_2$ with exactly one antiface.

The precise number of antifaces in $\Phi_2$ follows from Observation \ref{diremb-relcd}(b).  The maximum genus conclusion follows because the number of antifaces is minimum.

\smallskip
\noindent (b) For this case we just provide an outline of the proof, as the details are straightforward but tedious to give in full.  A $2$-vertex eulerian digraph $D$ with a $2$-edge-cut can be decomposed into two $1$-vertex eulerian digraphs $D_1$ and $D_2$ using a \emph{$2$-edge-cut reduction.}  There is a bijection between oriented directed embeddings $\Phi$ of $D$ and ordered pairs $(\Phi_1, \Phi_2)$ of oriented directed embeddings of $D_1$ and $D_2$.  (This generalizes \cite[Observation 3.3]{EE-M2mod4}.)  Applying this idea and part (a) gives the result.
\end{proof}

For Theorem \ref{densemain} we try to create situations in locally irreducible embeddings where we have interlacing, so that we can use the Interlaced Faces Lemma (\ref{interlacing}).
We wish to concisely describe sets of vertices that belong to specific antifaces.
For distinct antifaces $A$ and $B$ we define
$AB = V(A) \cap V(B)$ and
$A\ov B = V(A) - V(B)$.
For $i \in \{0,1\}$ we let $Ai$ denote the set of vertices that belong to $A$ and exactly $i$ other antifaces.
If $u \in AB$ or $u \in \only A$ we also say $u$ is of \emph{type} $AB$ or $\only A$, respectively.
If $u \in A\ov B$ or $u \in \also A$ we say $u$ is of \emph{general type} $A\ov B$ or $\also A$, respectively.
If the embedding is locally irreducible,
then $V(A) = \only A \cup \also A$,
$\also A = \bigcup_{P \in \cA, P \ne A} AP$,
and each vertex $u$ has a unique type which specifies exactly which antifaces $u$ belongs to.
In that case, for $u \in AB$ and $v \in AC$ where $A, B, C$ are distinct antifaces, every arc of $D$ between $u$ and $v$ must belong to $A$.

We have two results that allow us to find interlacing in locally irreducible embeddings.
For both proofs, note that two vertices are interlaced on $W$ if and only if they are interlaced on the walk $\usg{W}$.

\begin{lemma}[Three Neighbor Lemma]\label{3adj}
Assume we have a locally irreducible embedding with antiface $A$.  Suppose there is nonempty $S \subseteq \also A$ such that each $v \in S$ is adjacent to at least three other vertices of $S$ of types different from $v$.
Then there are vertices $x, y \in S$ of different types that are interlaced on $A$.
\end{lemma}

\begin{proof}
The hypothesis means that $|S| \ge 4$.  
Two adjacent vertices of $S$ of different types must be joined by an arc of $A$, so the hypothesis also means that each vertex of $S$ is incident with at least three edges of $\usg{A}$ (the walk in $\usg{D}$ corresponding to $A$), and so occurs at least twice on $\usg{A}$.

Therefore, we may choose an interval (sequence of consecutive vertices) $I=v_0 v_1 \dots v_p$ along $\usg{A}$ such that $v_0$ and $v_p$ are occurrences of some vertex $x \in S$, some $v_i$ with $1 \le i \le p-1$ is a vertex $y \in S$ of type different from $x$, and $p$ is as small as possible over all choices of $I$, $x$, and $y$.

If $y$ appears on $\usg{A}$ outside of $I$, then $x$ and $y$ are interlaced on $\usg{A}$ and hence on $A$ as desired.  So suppose that all occurrences of $y$ on $\usg{A}$ are in the interval $I$.  Let $v_m$ and $v_n$ be the first and last occurrences of $y$, so that $1 \le m < n \le p-1$.  No vertex $v_i$ with $m+1 \le i \le n-1$ can be a vertex of $S$ with type different from $y$, or we could replace $x$ by $y$ and reduce $p$. 
But then $y$ is adjacent to at most two vertices in $S$ of type different from itself, namely $v_{m-1}$ and $v_{n+1}$, which contradicts our hypothesis.
\end{proof}

\begin{lemma}[Diamond Lemma]\label{diamond}
Assume we have a locally irreducible embedding with distinct antifaces $A, B, C$.  Suppose that we have three distinct type $AC$ vertices $t, u, v$ and one type $AB$ vertex $x$, such that $tu, uv \in E(\usg{A})$ and  $xt, xu, xv \in E(\usg{D})$.
Then $x$ is interlaced on $A$ with at least one of $t$, $u$ or $v$.
\end{lemma}

\begin{proof}
Since $x$ is of type $AB$ and $t$, $u$, and $v$ are of type $AC$ with $B \ne C$, the arcs of $D$ joining $x$ to $t$, $u$, and $v$ must belong to $A$, so $xt, xu, xv \in E(\usg{A})$.
Therefore, possibly after reversing $\usg{A}$, we may choose an interval $I= v_0 v_1 \dots v_p$ along $\usg{A}$, where $v_0=v_p=x$, $v_1=u$, and no $v_i$, $1 \le i \le p-1$, is equal to $x$.
At most one of $xt$ or $xv$ belongs to $I$, so we may assume without loss of generality that $xt$, and hence $t$, occurs on $\usg{A}$ outside of $I$.
If $tu$ is an edge of $I$ then $t$ and $x$ are interlaced on $\usg{A}$ and hence on $A$, and if $tu$ is not an edge of $I$ then $u$ and $x$ are interlaced on $\usg{A}$ and hence on $A$.
\end{proof}

\subsection{Lemmas using density}

Now we give some consequences of the above results that depend on various degree conditions.
To state these results we define $k$ to be the maximum degree of the complement of $\usg{D}$, or equivalently $k = n-1-\delta(\usg{D})$.
In other words, for each $v \in V(D)$ there are at most $k$ vertices different from $v$ that are not adjacent to $v$ in $D$.
The parameter $k$ will have this meaning for the rest of this subsection and in Section \ref{MainProof}.
Thus, if $S$ is a set of at least $k+1$ vertices, any vertex $v$ not in $S$ is adjacent to, and hence shares an antiface with each of, at least $|S|-k \ge 1$ elements of $S$.  Moreover, the induced subgraph $\usg{D}[S]$ has minimum degree at least $|S|-k-1$.

We will use the following observation frequently, often without explicit reference.

\begin{observation}\label{loopfacebig}
If $A \in \cA$ has $\only A \ne \emptyset$ then $|V(A)| \ge n-k$.
\end{observation}

\begin{proof}
If $v \in \only{A}$ then $v$ and all its neighbors are in $A$.  Thus, at most $k$ vertices are not in $A$, i.e., $|V(A)| \ge n-k$.
\end{proof}

The next four results give conditions under which we can employ the Three Neighbor Lemma (\ref{3adj}) or the Diamond Lemma (\ref{diamond}) to find interlacing.

\begin{corollary}[Three Neighbor Corollary]\label{3adjcor}
Suppose we have a locally irreducible embedding with at least three antifaces.  Suppose that for some antiface $A$ we have $|\also A|-|AP| \ge k+3$ for all antifaces $P \ne A$.  Then there are vertices $x, y$ both of general type $\also A$ but of different types that are interlaced on $A$.
\end{corollary}

\begin{proof}
Let $S = \also A$.
Consider an arbitrary $v \in S$, of type $AP$.  Since $|\also A|-|AP| \ge k+3$ there are at least $k+3$ vertices of type different from $v$ in $S$.  Since $v$ is nonadjacent to at most $k$ of these, $v$ is adjacent to at least three of them.  The result then follows from the Three Neighbor Lemma (\ref{3adj}).
\end{proof}

\begin{corollary}[Big and Moderate Faces Corollary]\label{BMFlemma}
Suppose we have a locally irreducible embedding with three distinct antifaces $A, B, C$ where $|V(A)| \ge n-k$ (for example, if $\only{A} \ne \emptyset$) and $|V(B)|, |V(C)| \ge 2k+3$.
Then there are a vertex of type $AB$ and a vertex of type $AC$ that are interlaced on $A$.
\end{corollary}

\begin{proof}
Since $|V(A)| \ge n-k$, at most $k$ vertices of $B$ are not also vertices of $A$, so $|AB| \ge k+3$.  Similarly, $|AC| \ge k+3$.  Each vertex of $AB$ is therefore adjacent to at least three vertices of $AC$, and vice versa, so we may apply the Three Neighbor Lemma (\ref{3adj}) with $S = AB \cup AC$.
\end{proof}

We will use the following general lemma to apply the Three Neighbor Lemma (\ref{3adj}) in one particular situation.
A graph is \emph{$d$-degenerate} if every subgraph has minimum degree at most $d$.

\begin{lemma}[Bipartite Degeneracy Lemma]\label{bipdegen}
Let $d$ be a nonnegative integer.
If a simple bipartite graph $G$ of order $n \ge 2d$ is $d$-degenerate then it has at most $d(n-d)$ edges.  Thus, a simple bipartite graph of order $n \ge 2d$ with more than $d(n-d)$ edges has a subgraph with minimum degree at least $d+1$.
\end{lemma}

\begin{proof}
Construct a sequence of graphs $G_0, G_1, G_2, \dots, G_{n-2d}$, where $G_0 = G$ and for $i \ge 1$, $G_i = G_{i-1} - v_i$ where $v_i$ has degree at most $d$ in $G_{i-1}$.  Since $G_{n-2d}$ is a bipartite simple graph of order $2d$, it follows that $|E(G_{n-2d})| \le d^2$, and $|E(G)| \le |E(G_{n-2d})| + d(n-2d) \le d^2 + d(n-2d) = d(n-d)$.
\end{proof}

This lemma is sharp:  if $n \ge 2d$, then $K_{d,n-d}$  is a $d$-degenerate bipartite graph with exactly $d(n-d)$ edges.

\begin{corollary}[Diamond Corollary] \label{diamondcor}
Suppose that we have a locally irreducible embedding with distinct antifaces $A$ and $B$.  Suppose that either $k=0$ and $|AB| \ge 3$, or $|AB| \ge 3k+4$, and that there exist vertices of types $AP$ and $BQ$ with $P, Q \notin \{A, B\}$ (possibly $P=Q$).
Then there is a vertex of type $AB$ that is either interlaced on $A$ with a vertex of type $AP$ or interlaced on $B$ with a vertex of type $BQ$.
\end{corollary}

\begin{proof}
Let $x$ be a vertex of type $AP$, and $x'$ a vertex of type $BQ$, and let $S$ be the set of vertices of type $AB$ adjacent to both $x$ and $x'$.  Then $|S| \ge |AB|-2k$.  Let $H$ be the subgraph of $\usg{D}$ induced by $S$.  Then $\delta(H) \ge |S|-k-1 \ge |AB|-3k-1$.  Every edge of $H$ is an edge of $\usg{A}$ or $\usg{B}$ or both.

If $k=0$ and $|AB| \ge 3$, then $\usg{D}$ is complete, which means that $|S| = |AB| \ge 3$ and $H$ has a triangle.
If $|AB| \ge 3k+4$ then $\delta(H) \ge 3$ so $H$ has a vertex of degree at least $3$.  In either case $H$ has three edges that pairwise have a vertex in common.  At least two of these edges, say $tu$ and $uv$, belong to the same one of $\usg{A}$ or $\usg{B}$.

If $tu, uv \in E(\usg{A})$ then we apply the Diamond Lemma (\ref{diamond}) to $t, u, v$ and $x$ to get $x$ interlaced with one of $t, u, v$ on $A$.  If $tu, uv \in E(\usg{B})$ then we apply the Diamond Lemma (\ref{diamond}) to $t, u, v$ and $x'$ to get $x'$ interlaced with one of $t, u, v$ on $B$.
\end{proof}

We could also try to improve the above result by finding a triangle in $H$ when $k \ge 1$.  To do this we could use Mantel's Theorem (an $n$-vertex simple graph with average degree greater than $n/2$ has a triangle) for $H$.  However, this would require $|AB| \ge 4k+3$, which is at least as strong a condition as $|AB| \ge 3k+4$ when $k \ge 1$.

In our arguments we need the antifaces to be reasonably large (have many vertices) so that we can force interlacing to occur using the Three Neighbor Lemma (\ref{3adj}) or the Diamond Lemma (\ref{diamond}) or their consequences.
The next result will be applied to use large antifaces to help `blow up' smaller antifaces in an embedding, to obtain situations where we can make progress.
It involves a cyclically ordered finite set of points with some possible break points (black), some reasonably uniformly distributed points (white) and some arbitrarily distributed points (red).  It allows us to split the order at two black points so that we obtain something close to a desired division of the white and red points.

For this result we need to work with cyclic orders, both of a finite set of points and of the points of a circle.  In both cases, for distinct points $a$ and $b$, we denote by $(a,b)$ the set of points encountered moving forward from $a$ until we reach $b$ (excluding $a$ and $b$); we can then define $[a,b] = (a,b) \cup \{a,b\}$.
We define $(a,a) = \emptyset$ and $[a,a] = \{a\}$.
For $c \in \mR$ we let $c\mZ = \{ ci \;|\; i \in \mZ\}$.
We treat a circle of circumference $c$ as the additive group $\mR/c\mZ$ of reals modulo $c$, and identify its points with the interval $[0,c) \subseteq \mR$.

\begin{lemma}[Division Lemma]\label{division}
Suppose $C$ is a finite cyclic order where each point of $C$ is colored black, white, or red.  Let the black points be $b_0, b_1, b_2, \dots, b_{k-1}$ in cyclic order (indexed by elements of $\mZ_k$), and let $W$ and $R$ be the sets of white and red points, respectively.
Let $m$ be a positive integer.
Suppose that $|W \cap [b_i, b_{i+1}]| \le m$ for every $i \in \mZ_k$, and that $|W| > |R|$.
Then for every $p \in \mR$ with $m \le p \le |W \cup R|-m$ there are $b_i$ and $b_j$, $i \ne j$, such that
$p-m < |(W \cup R) \cap [b_i,b_j]| < p+m$. 
\end{lemma}

\begin{proof}
Let $\nw_i = |W \cap [b_i,b_{i+1}]| \le m$ and $\nr_i = |R \cap [b_i, b_{i+1}]|$ for $i \in \mZ_k$.  Let $\ell = |W \cup R| = |W|+|R| = \sum_{i \in \mZ_k} (\nw_i+\nr_i)$.
Because $m \le p \le \ell-m$, we have $\ell \ge 2m$.  Since $|W| > |R|$, we have $|W| \ge m+1$, and so $k \ge 2$.
Define (not necessarily distinct) real numbers $x_0, x_1, \dots, x_k$ by
$$x_0 = 0, \qquad\hbox{and}\qquad
 x_{i+1} = x_i + \nw_i + \nr_i \quad\hbox{for $i \in \{0, 1, 2, \dots, k-1\}$}.$$
Then $x_k = \ell$, so we may identify $x_0=0$ with $x_k = \ell$ in $[0,\ell] \subseteq \mR$ and treat the resulting set $C'$ simultaneously as the group $\mR/\ell\mZ$, as a cyclic order, and as a circle of circumference $\ell$.
We use $\la$ to denote the length measure on $C'$.

We now color each point of $C'$ black, white, or red.
All points $x_i$ are black.
For each $i \in \mZ_k$ the interval $(x_i, x_{i+1})$ is empty if $\nw_i=\nr_i=0$, but otherwise we color it as follows.
If $\nw_i=0$ and $\nr_i > 0$ we color all of $(x_i, x_{i+1})$ red.  If $\nw_i > 0$ then we color $(x_i, x_i+\hf \nw_i) \cup (x_{i+1}-\hf \nw_i, x_{i+1})$ white and $[x_i + \hf \nw_i, x_{i+1}-\hf \nw_i]$ red (even if $\nr_i=0$, the point halfway between $x_i$ and $x_{i+1}$ is red, which we need for (iii) below).  Let $B'$, $W'$, and $R'$ be the sets of black, white, and red points in $C'$, respectively.  We note the following properties.

% see https://tex.stackexchange.com/questions/119919/
%    no-spacing-between-enumerated-items-with-usepackageenumerate
% \begin{enumerate}[label=(\roman*),topsep=2pt,parsep=0pt]
\begin{enumerate}[(i),topsep=2pt,parsep=0pt]

\item $B'$ is a set of measure zero.

\item $\la([x_i, x_{i+1}]) = \nw_i + \nr_i = |(W \cup R) \cap [b_i, b_{i+1}]|$, with
$\la(W' \cap [x_i, x_{i+1}]) = \nw_i$, and
$\la(R' \cap [x_i, x_{i+1}]) = \nr_i$.

\item If $y \in B' \cup W'$ then there exists $i$ with $x_i \in (y-\hf m, y+\hf m)$.
\end{enumerate}

From (i) and (ii) it follows that $\ell = \la(C') = |R|+|W|$, with $\la(W') = |W|$, and $\la(R') = |R| < |W|$. Hence $\la(R') < \hf \la(C')$.
Therefore, letting $R''$ be $R'$ translated by $-p$ in $C' = \mR/\ell\mZ$, we have
$\la(R' \cup R'') \le \la(R') + \la(R'') = 2\la(R')< \la(C')$, so there exists $y \notin R' \cup R''$, which means that $y, y+p \notin R'$.  By (iii) there is some $x_i \in (y-\hf m,y+\hf m)$ and some $x_j \in (y+p-\hf m, y+p+\hf m)$; these intervals are disjoint because $m \le p \le \ell-m$, so $x_i \ne x_j$.  Now $\la([x_i,x_j])$ differs by less than $m$ from $\la([y,y+p])=p$.  Again from (i) and (ii) we get $\la([x_i,x_j]) = |(W \cup R) \cap [b_i, b_j]|$, so our result follows.
\end{proof}

We will apply the Division Lemma (\ref{division}) with $p$ either an integer or a half-integer.  In these cases the bounds on $q=|(W \cup R) \cap [b_i, b_j]|$, which say that $|q-p| < m$, cannot be improved.  When $p$ is an integer, the bound becomes $|q-p| \le m-1$, and when $p$ is a half-integer it becomes $|q-p| \le m-\hf$.
Consider $\mZ_{4m+1}$, with $m \ge 1$, where $B = \{b_0=0, b_1=m+1\}$, and suppose $[b_0, b_1]$ contains $m$ white points, and $[b_1, b_0]$ contains $m$ white and $2m-1$ red points.
For $p=2m-1$ we must use $|(W \cup R) \cap [b_0,b_1]| = m = p-(m-1)$ to satisfy the conclusion, and for $p=2m$ we must use $|(W \cup R) \cap [b_1,b_0]| = 3m-1 = p+m-1$.
These show that the bound $|q-p| \le m-1$ cannot be improved when $p$ is an integer.
For $p=2m-\hf$ we must use either $|(W \cup R) \cap [b_0, b_1]| = m = p-(m-\hf)$ or $|W \cup R) \cap [b_1, b_0]| = 3m-1 = p + m-\hf$.
This shows that the bound $|q-p| \le m-\hf$ cannot be improved when $p$ is a half-integer.

We also cannot relax the condition $|W| > |R|$ to $|W| \ge |R|$.
Consider $\mZ_{6m+3}$, $m \ge 1$, with $B=\{b_0=0,b_1=2m+1,b_2=4m+2\}$, where each interval $[b_i,b_{i+1}]$ contains $m$ white and $m$ red points, so that $|W|=|R|=3m$.  If we let $p=3m$, then we cannot achieve $2m=p-m < |(W \cup R) \cap [b_i,b_j]| < p+m=4m$.

We now use the Division Lemma (\ref{division}) to show that we can use a large face $A$ to increase the size of a smaller face $B$.

\begin{lemma}[Blow Up Lemma]\label{blowup}
Suppose we have a locally irreducible embedding $\Phi$ with distinct antifaces $A$ and $B$, and there exists a vertex $x$ of type $AB$.  
Suppose that $|V(A)| \ge 5$ and $n \ge k+3$.
Then there is an embedding $\Phi'$ (possibly $\Phi'=\Phi$) that has the same profaces and antifaces as $\Phi$ except that $A$ and $B$ are replaced by antifaces $A'$ and $B'$ where $|V(A')|, |V(B')| \ge
\min(\hf |V(A)|-1, \hf(n-k-1))$ and $x$ has type $A'B'$.
\end{lemma}

\begin{proof}
Let $a = |V(A)|$ and $b = |V(B)|$.  Suppose that the sequence of vertices along $A$ is $(v_0, v_1, v_2, \dots, v_{t-1})$, with indexing by elements of $\mZ_t$.  Let $Y = \{y \ne x\;|\; y = v_{i-1}$ or $v_{i+1}$ for some $i \in \mZ_t$ with $v_i = x\}$; in other words, $Y \subseteq V(A)$ is the set of vertices joined to $x$ by at least one arc of $A$.
We know $x$ has at least $n-k-1$ neighbors, each of which must be joined to $x$ by an arc of $A$ or $B$ (or both). A vertex joined to $x$ by an arc of $B$ must be a vertex of $B$, and $x$ has at most $b-1$ neighbors in $V(B)-\{x\}$.  Thus, $|Y| \ge (n-k-1)-(b-1) = n-k-b$.
Let $Z = V(A) - (\{x\} \cup Y)$.

We will color some of the elements of $\mZ_t$ as follows.  If $v_h = x$ then $h$ is black.
For each $y \in Y$ choose one arc $e_y$ of $A$ joining $x$ and $y$, let $v_i$ be the occurrence of $y$ on $A$ as an end of $e_y$, and color $i$ white.
For each $z \in Z$ choose one occurrence $v_j$ of $z$ on $A$ and color $j$ red.
So there are $|Y|$ white points and $|Z|$ red points.
There are at most two arcs $e_y$ between consecutive occurrences of $x$ on $A$, and hence at most two white points of $\mZ_t$ between consecutive black points.

Suppose first that $|Y| > |Z|$.
Discard all uncolored points of $\mZ_t$, creating a cyclic order $C$ with some black points, $\ell=a-1 \ge 4$ white or red points, and more white than red.
Apply the Division Lemma (\ref{division}) with $m=2$ and $p=\hf\ell$: since $\ell \ge 4$ we have $2 \le p \le \ell-2$.  We obtain distinct black points $h, j \in C$ so that the interval $[h,j]$ of $C$ contains $q_1$ white or red points where $p-2 < q_1 < p+2$.  The number of white or red points in $[j,h]$ is $q_2 = \ell-q_1 = 2p-q_1$, which also satisfies $p-2 < q_2 < p+2$.  Since $p, q_1, q_2 \in \hf \mZ$ we get $q_1, q_2 \ge p-\frac32 = \hf (\ell-3)$.
Write $A$ (after cyclic shifting) as $A_1 \mw A_2$ where $A_1 = v_h(=x) \dots v_j(=x)$ and $A_2 = v_j \dots v_h$.  Each $A_i$ contains $q_i$ distinct points of $V(A)-\{x\}$ and hence $q_i+1 \ge \hf (\ell-1)$ distinct points of $V(A)$.
Thus, applying part (b) of the Three Face Lemma (\ref{3face}) to $A = A_1 \mw A_2$ and $B$ gives new antifaces $A'$ and $B'$ each with at least $\hf (\ell-1) = \hf a -1$ vertices.

Now suppose that $|Y| \le |Z|$.  Let $c = \min(\hf a-1, \hf(n-k-1))$.  Since $a \ge 5$, $a \ge \hf a-1 \ge c$.  If $b \ge \hf(n-k-1)$ then $b \ge c$ as well, and we take $A'=A$, $B'=B$.
So we may assume that $b \le \hf(n-k-2)$.  Now $|Y| \ge n-k-b \ge \hf(n-k+2)$ and therefore $|Y| \ge 3$ since $n \ge k+3$ and $|Y|$ is an integer.
Uncolor some points of $Z$ so that the number of red points is $|Y|-1$, which is less than the number of white points.  Apply the Division Lemma (\ref{division}) (discarding uncolored points) and part (b) of the Three Face Lemma (\ref{3face}) as in the previous paragraph, but with $\ell = 2|Y|-1 \ge 5 \ge 4$ white or red points.  This gives new antifaces $A'$ and $B'$ each with at least $\hf (\ell-1) = |Y|-1 \ge \hf(n-k)$ vertices.

Thus, we always have $A'$ and $B'$ each with at least $\min(\hf a-1, \hf(n-k-1))$ vertices.  In all cases $x \in V(A') \cap V(B')$.
\end{proof}

\section{Proof of the main result} \label{MainProof}

In this section we prove Theorem \ref{densemain}.  We need one final tool, namely the touch graph of a locally irreducible embedding.  The touch graph encodes information about how each antiface in in a locally irreducible embedding touches itself and other antifaces at the vertices, and  will make our arguments easier to explain and visualize.

\begin{definition}
Let $\Phi$ be a locally irreducible oriented directed embedding.  The \emph{touch graph} $K$ is an undirected graph which may have loops and multiple edges.  The vertices of $K$ are the antifaces of $\Phi$,  and the edges of $K$ are the vertices of $D$, i.e., $V(K) = \cA$ and $E(K) = V(D)$. For the incidence relation of $K$, a vertex of $D$ of type $\only A$ in $\Phi$ becomes a loop of $K$ at $A$, and a vertex of type $AB$ becomes a link (non-loop edge) of $K$ joining $A$ and $B$.
Thus in $K$, $A\ov B$ is the set of edges incident with $A$ but not $B$, $\only A$ is the set of loops at $A$, $\also A$ is the set of links at $A$, and
$V(A)$ is $E_K(A)$, the set of all edges incident with $A$ in $K$.  Since $\Phi$ is locally irreducible, $K$ is well-defined.

It will sometimes be convenient to represent $K$ in simplified form, where we replace multiple parallel edges (including loops) by a single edge weighted by the number of parallel edges.
\end{definition}

Critically, in the touch graph, all sets of vertices of $D$ are interpreted as sets of edges.  Thus, symbols such as $u, v, w, \dots$ represent vertices in $D$ but edges in $K$.  See Figures \ref{touchgraphbig}, \ref{MergeGreenFaces}, and \ref{touchgraphJo}.

Figure \ref{touchgraphbig} shows a directed embedding of a digraph $D$, where $D$ is an orientation of $K_{12}$ minus a perfect matching.  It has two profaces, one dark teal (dark, if color is not visible) and one light teal (light), shown on the left.  It has six anti-faces, one light yellow (very light), one orange (light), one dark red (dark), and three black, shown on the right. It is straightforward to check from the rotations at the vertices that this is indeed a directed embedding.

\begin{figure}[ht!]
     \centering
\hfill
   \begin{subfigure}[c]{0.45\textwidth}
        \centering
    \includegraphics[clip, trim=1cm 11.5cm 2cm 1cm,scale=0.5]{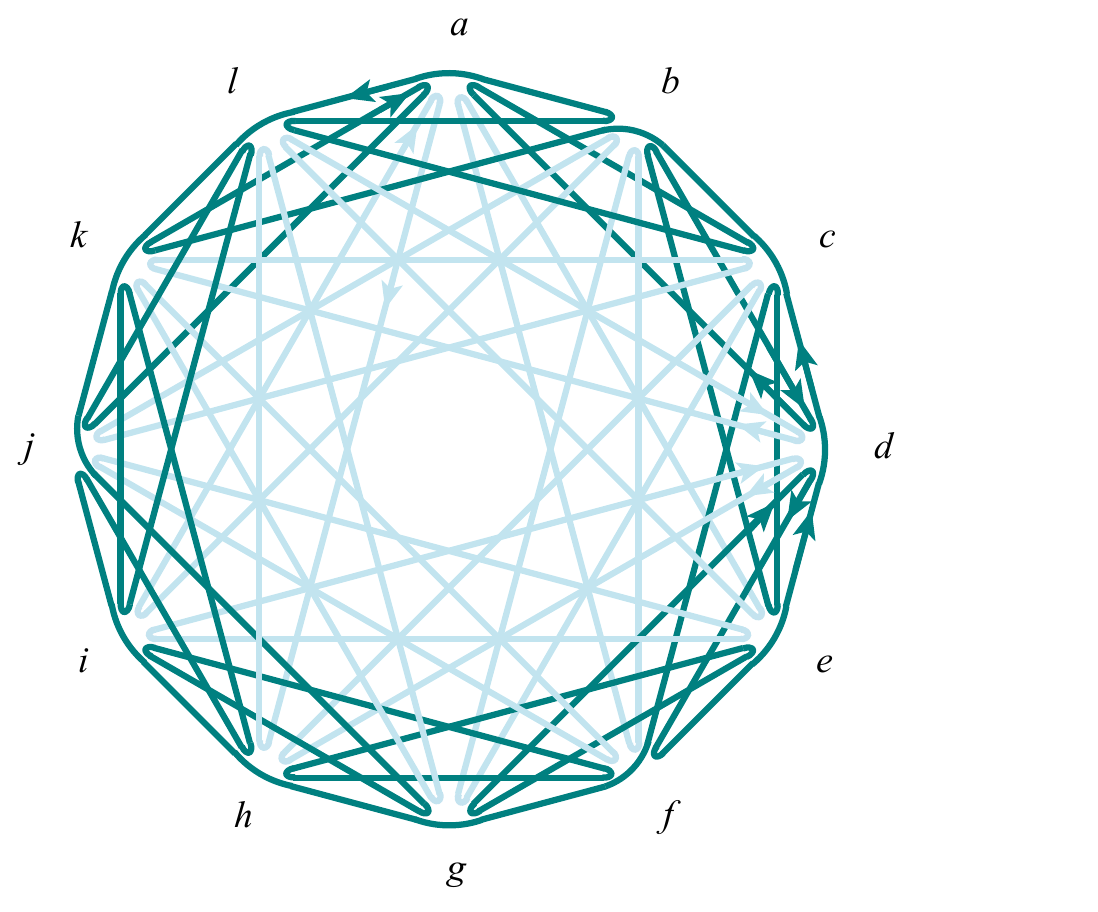}
        \caption{Profaces}
        \label{profaces}
     \end{subfigure}
        \hfill
        \begin{subfigure}[c]{0.475\textwidth}
        \centering
\centering
    \includegraphics[clip, trim=2cm 3.25cm 2cm 2.5cm,scale=0.63]{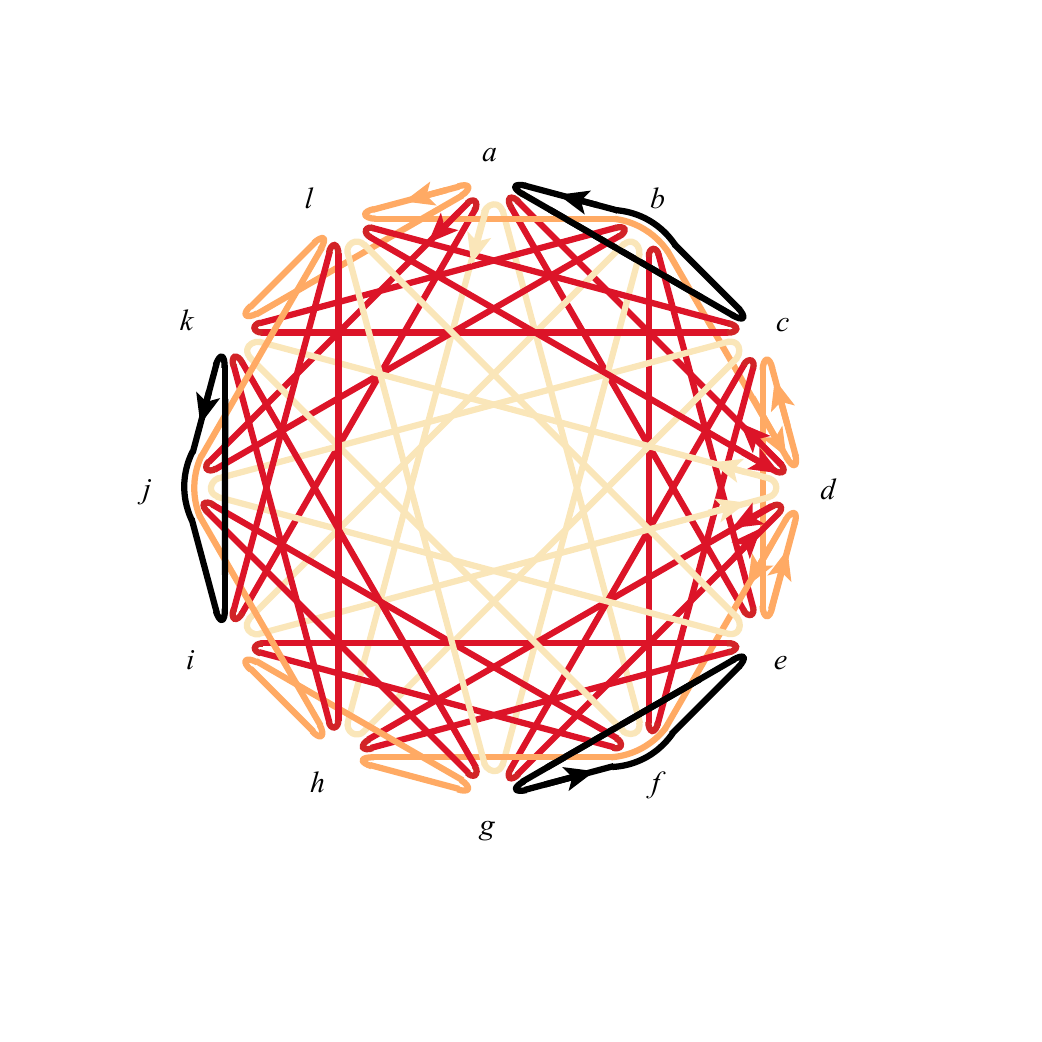}
        \caption{Antifaces}
        \label{antifaces}
     \end{subfigure}
\caption{A directed embedding of an  orientation of $K_{12}$ minus a perfect matching.}
\label{touchgraphbig}
\end{figure}

  Notice that the three antifaces colored light yellow, orange, and dark red, meet at vertex $d$ in the embedding of $D$.  Thus we can apply Lemma \ref{3face} to merge these faces into a single orange face, shown in Figure \ref{MergeGreenFaces}.   After merging these faces, the embedding of $D$ has only four antifaces and is locally irreducible.  We can thus create the touch graph shown in Figure \ref{touchgraphJo}.

\begin{figure}[h!]
  \centering    \includegraphics[clip, trim=2cm 3cm 3cm 2cm, width=0.5
  \textwidth]{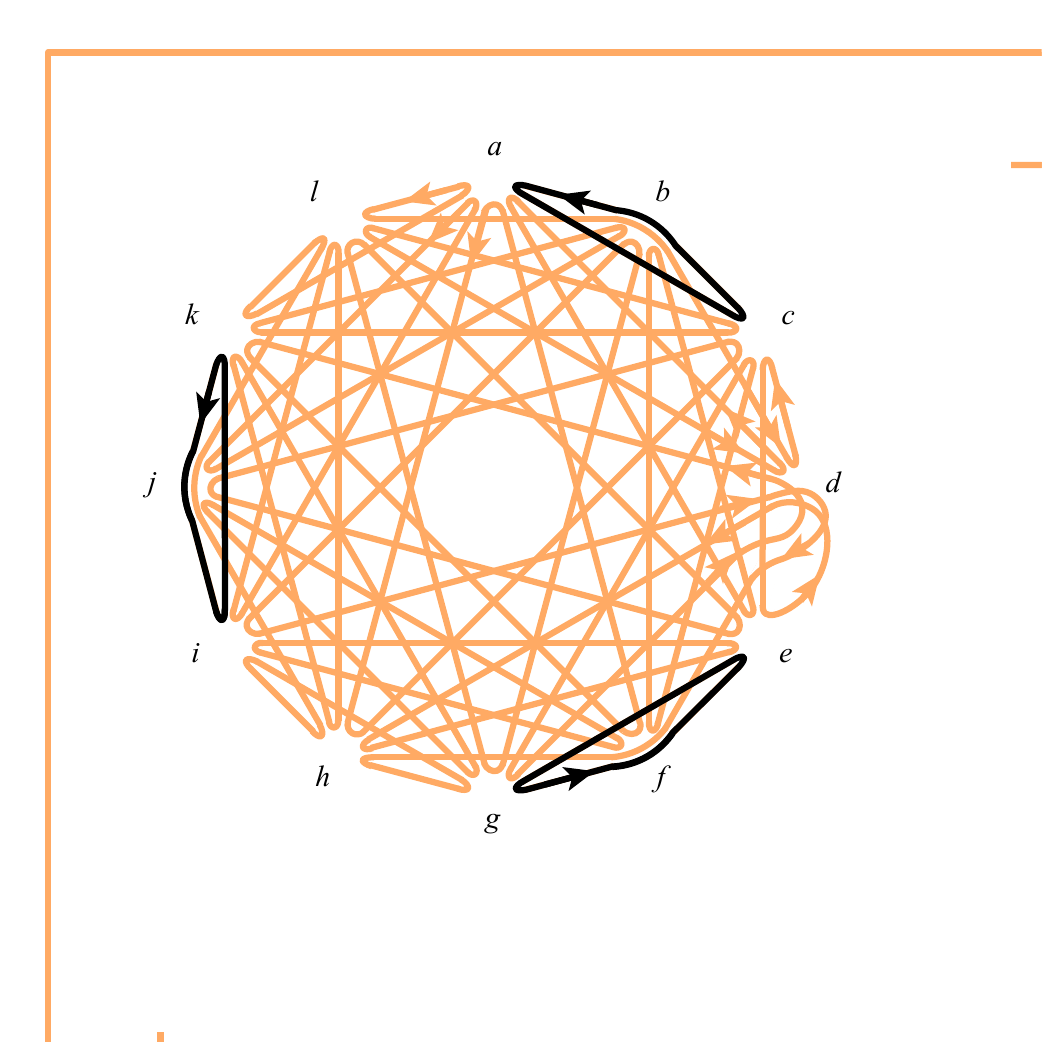}
 \caption{Applying the 3-face lemma to the light yellow, orange, and dark red faces at vertex $d$ to create only a single orange face. } 
\label{MergeGreenFaces}
  \end{figure}

\begin{figure}
  \hbox to \textwidth{%
    \hfill
    \includegraphics[width=0.3\textwidth]{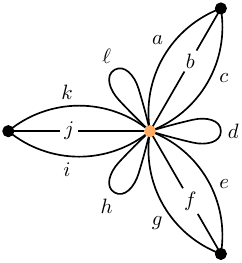}%
    \hfill
    \includegraphics[width=0.3\textwidth]{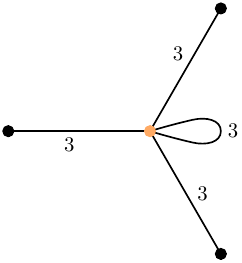}%
    \hfill
  }
  \caption{Touch graph for the three black and one orange (light) antifaces of Figure~\ref{MergeGreenFaces}, with simplified weighted version at right.} 
  \label{touchgraphJo}
\end{figure}

Note that $K$ is connected when $D$ is connected.  This can be seen as follows.
Given $A, B \in \cA$ choose $v \in V(A)$ and $w \in V(B)$.  The sequence of vertices on a $vw$-path in $\usg{D}$ gives a sequence of adjacent edges in $K$ with the first, $v$, adjacent to $A$ and the last, $w$, adjacent to $B$, and so $K$ has a connected subgraph containing $A$ and $B$.

We now have all the tools we need to prove the main theorem of this paper, which we repeat here for convenience.

\begin{theorem-densemain}
\densemaintext
\end{theorem-densemain}

\begin{proof}[Proof of Theorem \ref{densemain}]
We have an eulerian digraph $D$ satisfying $\delta(\usg{D}) \ge (4n+2)/5$.
In terms of $k = n-1-\delta(\usg{D})$, this means that $n \ge 5k+7$.
By our definition of eulerian, $D$ is connected (this also follows from $\delta(\usg{D}) \ge n/2$). 

We are given a directed circuit decomposition $\cC$ of $D$.  By Lemma \ref{dcdembedding} there exists an oriented directed embedding $\Phi$ of $D$ whose set of profaces is $\cC$.  As usual, we let $\cA$ denote the set of antifaces.  We will show that if $|\cA| \ge 3$ then we can reduce the number of antifaces, without changing the profaces.  After sufficiently many reductions, we obtain an embedding with one or two antifaces as desired. We use a case analysis, as follows.
In Case 1 we dispose of situations where $\Phi$ is not locally irreducible.  For the remaining cases we consider the structure of the touch graph $K$.  Case 2 is when $K$ has no loops, with subcases 2.1 and 2.2 depending on whether $K$ is not or is a star, respectively.  Case 3 is when $K$ has loops, with subcases 3.1 and 3.2 depending on whether there are loops at two or more vertices or at just one vertex, respectively.

\setcounter{case}{0}
\begin{case}\label{reduce}
Suppose some vertex belongs to three distinct antifaces.
Then applying part (a) of the Three Face Lemma (\ref{3face}) reduces the number of antifaces.
\end{case}

\smallskip
So we may henceforth assume that Case 1 does not happen, i.e., that $\Phi$ is locally irreducible.  We let $K$ be the touch graph of $\Phi$.   We say $K$ is a \emph{star} if $K$ has one vertex $A$ with which every edge is incident; in other words, if $V(A) = E_K(A)$ is equal to $V(D) = E(K)$ for some $A \in \cA$.

\begin{case}\label{noloops} We begin by assuming that $K$ has no loops.

\end{case}

\begin{subcase}\label{noloopnostar}
Suppose $K$ has no loops and is not a star, i.e., $\only P = \emptyset$ and $V(P) \ne V(D)$ for all $P \in \cA$.
Then $\also P = V(P)$ for all $P \in \cA$.
Let $A, B \in \cA$ be such that $\max \{|PQ| \;|\; P, Q \in \cA, P \ne Q\} = |AB|$, i.e., $A$ and $B$ have the maximum number of common vertices over all pairs of circuits in $\cA$. Without loss of generality assume that $|V(A)| \ge |V(B)|$, which implies that $|A \ov B| \ge |B\ov A|$.

Our strategy in Case \ref{noloopnostar} is to use the Three Neighbor Corollary (\ref{3adjcor}) (or a similar argument based on the Bipartite Degeneracy Lemma (\ref{bipdegen})) if $|AB|$ is not too large, or the Diamond Corollary (\ref{diamondcor}) if it is.
Because of the structure of our argument, we deal with small $|AB|$ first, and then the remaining situations in decreasing order of $|AB|$.

We know that every edge of $K$ is adjacent to (shares a vertex with) all but at most $k$ of the other edges.  If $S \subseteq E(K)$ with $|S| \ge k+1$, then every edge of $K$ not in $S$ is adjacent to at least $|S|-k$ edges of $S$, and hence is incident with an end of some edge of $S$.

\begin{subsubcase}\label{lek}
Suppose that $|AB| \le k$.
There is some edge $v$ in $K$ of type $AB$, and at most $k$ edges are not adjacent to $v$, i.e., at most $k$ edges of $K$ are incident with neither $A$ nor $B$.  Thus, $|V(A) \cup V(B)| \ge n-k \ge 4k+7$.  Since $|V(A)| \ge |V(B)|$, we have $|\also A| = |V(A)| \ge 2k+4$.
Since $|AP| \le |AB| \le k$ for all $P \in \cA$ with $P \ne A$, we have $|\also A|-|AP| \ge k+4$ for all such $P$.  Thus, there is interlacing by the Three Neighbor Corollary (\ref{3adjcor}), and by the Interlaced Faces Lemma (\ref{interlacing}) we can reduce the number of antifaces. (This argument also works for $|AB|=k+1$, although we cover that situation in Case \ref{mediumAB}.)
\end{subsubcase}% <= k

In the remainder of Case \ref{noloopnostar}, we may suppose that $|AB| \ge k+1$.  Thus, every vertex not in $AB$ must be adjacent to at least one vertex in $AB$, and hence every edge of $K$ not in $AB$ is incident with either $A$ or $B$.
Thus, $V(D) = E(K) = V(A) \cup V(B) = A\ov B \cup AB \cup B\ov A$.
Since $A\ov B = V(D)-V(B)$ and $V(B) \ne V(D)$ and $K$ is not a star, we have $A \ov B \ne \emptyset$.  Similarly, $B\ov A \ne \emptyset$.
Figure \ref{touchgraph2} shows the (simplified weighted) structure of $K$ in this situation.

\begin{figure}
  \centering
  \includegraphics[scale=1.1]{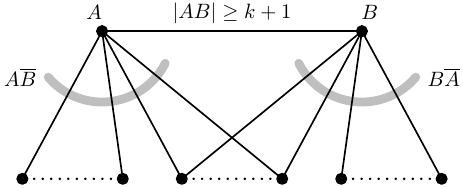}
  \caption{Touch graph for Cases \ref{bigAB}, \ref{eqAB}, and \ref{mediumAB}.}
  \label{touchgraph2}
\end{figure}

\begin{subsubcase}\label{bigAB}
Suppose that $|AB| \ge 3k+4$.  
Since $A\ov B, B\ov A \ne \emptyset$, there are $v \in AP$, $w \in BQ$ with $P, Q \notin \{A,B\}$.  So by the Diamond Corollary (\ref{diamondcor}) there is interlacing, and we can use the Interlaced Faces Lemma (\ref{interlacing}) to reduce the number of antifaces.
\end{subsubcase}

\begin{subsubcase}\label{eqAB}
Suppose that $|AB| = 3k+3$.  If $k=0$, then we can use the Diamond Corollary (\ref{diamondcor}) and the Interlaced Faces Lemma (\ref{interlacing}) to reduce the number of antifaces as in Case \ref{bigAB}.
Therefore, we may assume that $k \ge 1$.

We have $n = |A\ov B| + |B\ov A| + |AB| = |A\ov B|+|B\ov A|+3k+3 \ge 5k+7$, so $|A\ov B|+|B\ov A| \ge 2k+4$.
Since $|A\ov B| \ge |B\ov A|$, we have $|A\ov B| \ge k+2$.

Now think of $A\ov B$ and $AB$ as vertex sets in $D$, and consider the bipartite subgraph $H$ of $\usg{D}$ consisting of $V(H) = V(A)$ and exactly the edges of $\usg{D}$ between $A\ov B$ and $AB$.  We claim that $H$ has a subgraph $H'$ with minimum degree at least $3$.  This will follow from the Bipartite Degeneracy Lemma (\ref{bipdegen}) with $d=2$ provided $|V(H)| \ge 4$ and $|E(H)| > 2(|V(H)|-2)$. We can satisfy these conditions as follows. 
We have $|V(H)| = |A\ov B|+|AB| = |A\ov B|+3k+3 \ge 4k+5 \ge 4$.  Every vertex of $A\ov B$ is adjacent in $\usg{D}$, and hence in $H$, to at least $|AB|-k = 2k+3$ vertices of $AB$.  Since $k \ge 1$, it follows that
\begin{align*}
|E(H)| - 2(|V(H)|-2) & \ge (2k+3)|A\ov B| - 2(|A\ov B|+3k+1)
	= (2k+1)|A\ov B| -6k-2 \\
	&\ge (2k+1)(k+2)-6k-2
	= 2k^2-k = (2k-1)k > 0 .
\end{align*}
Thus, $E(H) > 2(V(H)-2)$ and we can apply the Bipartite Degeneracy Lemma (\ref{bipdegen}) with $d=2$. 
Therefore $H'$ exists, and we can apply the Three Neighbor Lemma (\ref{3adj}) with $S = V(H')$ to show that there is interlacing, and then use the Interlaced Faces Lemma (\ref{interlacing}) to reduce the number of antifaces.
\end{subsubcase}

\begin{subsubcase}\label{mediumAB}
Suppose that $k+1 \le |AB| \le 3k+2$.
Then $|A\ov B|+|B\ov A|=n-|AB| \ge (5k+7)-(3k+2)=2k+5$.  Since $|A\ov B| \ge |B\ov A|$, we have $|A\ov B| \ge k+3$.  Now for any $P \ne A$ we have $|\also A|-|AP| \ge |\also A|-|AB| = |V(A)|-|AB| = |A\ov B| \ge k+3$.  Therefore, by the Three Neighbor Corollary (\ref{3adjcor}) there is interlacing, and we can use the Interlaced Faces Lemma (\ref{interlacing}) to reduce the number of antifaces.
\end{subsubcase}

\end{subcase}

In the remaining cases, Cases \ref{noloopstar} and \ref{loop} below, we often need to modify the embedding to increase the number of vertices of some faces, using the Blow Up Lemma (\ref{blowup}).
We can always apply the Blow Up Lemma (\ref{blowup}) to $A, B \in \cA$ if $|V(A)| \ge n-k$ and $D$ has a vertex $x$ of type $AB$, because then $|V(A)| \ge 4k+7 \ge 5$, and we know that $n \ge 5k+7 \ge k+3$.
We will say we are \emph{blowing up $B$ using $A$}.

When we apply the Blow Up Lemma (\ref{blowup}), the number of antifaces does not change.
The two new antifaces $A'$ and $B'$ that replace $A$ and $B$ will have $|V(A')|, |V(B')| \ge \min(\hf(n-k)-1, \hf(n-k-1)) = \hf(n-k)-1 \ge \hf(4k+7)-1 = 2k+2\hf$, but since these are integers, $|V(A')|, |V(B')| \ge 2k+3$.
We therefore create two faces satisfying the `moderate face' condition of the Big and Moderate Faces Corollary (\ref{BMFlemma}).
Also we know that $V(A')\cup V(B') = V(A)\cup V(B)$, and $x \in A'B'$, so that $A'B' \ne \emptyset$.

\begin{subcase}\label{noloopstar}
Suppose that $K$ has no loops and is a star, i.e., $\only P = \emptyset$ for all $P \in \cA$ and $V(A) = V(D) = E(K)$ for some $A \in \cA$.
Then $AP = \also P = V(P)$ for all $P \in \cA$, $P \ne A$.
Let $B \in \cA$ be such that $\max \{|AP| \;|\; P \in \cA, P \ne A\} = |AB|$.

Our strategy in Case \ref{noloopstar} is to use the Three Neighbor Corollary (\ref{3adjcor}) if $|AB|$ is not too large, but if it is, we need to blow up a third face and use the Big and Moderate Faces Corollary (\ref{BMFlemma}).

If $|\also A|-|AB| \ge k+3$ then since $|AP| \le |AB|$ we have $|\also A|-|AP| \ge k+3$ for all $P \ne A$, there is interlacing by the Three Neighbor Corollary (\ref{3adjcor}), and by the Interlaced Faces Lemma (\ref{interlacing}) we can reduce the number of antifaces.  Therefore, we may assume that $|A\ov B| = |\also A|-|AB| = n-|AB| \le k+2$, i.e., that $|V(B)| = |AB| \ge n-k-2$.  Thus, at most $k+2$ vertices of $D$ do not belong to $B$.
The left side of Figure \ref{touchgraph34} shows the structure of $K$ in this situation.

There must be some third face $C$, so $|AC| = |V(C)| > 0$.
Since $|V(A)| = n$, we can blow up $C$ using $A$, obtaining a new embedding $\Phi'$ (and touch graph $K'$) with set of antifaces $\cA'$ in which $A, C$ are replaced by $A', C'$ with $|V(A')|, |V(C')| \ge 2k+3$ and $A'C' \ne \emptyset$.
If $\Phi'$ is not locally irreducible, we may apply Case \ref{reduce} to reduce the number of antifaces, so we may assume that $\Phi'$ is locally irreducible.

Since at most $k+2$ vertices of $D$ do not belong to $B$, we have $|BA'| \ge |V(A')| - (k+2) \ge k+1$ and similarly $|BC'| \ge k+1$. Also, $A'C' \ne \emptyset$.  Thus, $K'$ contains at least one triangle $(A'BC')$ and hence is not a star.

Assume without loss of generality that $|\only{A'}| \ge |\only{C'}|$.  
If $|\only{A'}|=0$ then $K'$ has no loops ($\only P$ has not changed for $P \ne A, C$), so we can apply Case \ref{noloopnostar} to reduce the number of antifaces.
So we may assume that $|\only{A'}| > 0$. Then $|V(A')| \ge n-k$ by Observation \ref{loopfacebig}, $|V(B)| \ge n-k-2 \ge 4k+5 \ge 2k+3$, and $|V(C')| \ge 2k+3$.  Thus, there is interlacing by the Big and Moderate Faces Corollary (\ref{BMFlemma}), and we can apply the Interlaced Faces Lemma (\ref{interlacing}) to reduce the number of antifaces.
\end{subcase}

\begin{figure}
  \hbox to \textwidth{%
    \hfill
    \includegraphics[scale=1.1]{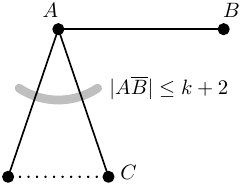}%
    \hfill\hfill
    \raise10pt\hbox{\includegraphics[scale=1.1]{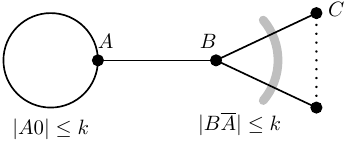}}%
    \hfill
  }
  \caption{Touch graph for Case \ref{noloopstar} with $|AB|$ large (left) and for Case \ref{onenhbr} (right).}
  \label{touchgraph34}
\end{figure}

\begin{case}\label{loop}
Suppose that $K$ has at least one loop, i.e. there is at least one $P \in \cA$ with $\only P \ne \emptyset$.
By Observation \ref{loopfacebig}, $|V(P)| \ge n-k$ for each such $P$.

Our strategy in Case \ref{loop} is to blow up faces until we either obtain a previous case or can apply the Big and Moderate Faces Corollary (\ref{BMFlemma}).

\begin{subcase}\label{twovert}
Suppose there are loops incident with at least two vertices of $K$.  Thus there are distinct $A,B \in \cA$ with both $\only A$ and $\only B$ nonempty.

Since $K$ is connected and $|\cA| \ge 3$, at least one of $A$ and $B$, say $A$, is adjacent to a third vertex $C$ of $K$, i.e., $AC \ne \emptyset$.  Since $|V(A)| \ge n-k$ we can blow up $C$ using $A$, to obtain a new embedding $\Phi'$ containing faces $B, A', C'$ with $|V(B)| \ge n-k$ and $|V(A')|, |V(C')| \ge 2k+3$.
If $\Phi'$ is not locally irreducible we apply Case \ref{reduce}, and otherwise we have interlacing by the Big and Moderate Faces Corollary (\ref{BMFlemma}) and use the Interlaced Faces Lemma (\ref{interlacing}), reducing the number of antifaces in either case.
\end{subcase}

\begin{subcase}\label{onevert}
Suppose there is exactly one vertex of $K$ with one or more incident loops. 

\begin{subsubcase}\label{onenhbr} Suppose that the one vertex with loops has only one neighbor in $K$.
In other words, $\only A \ne \emptyset$ for a unique $A \in \cA$ and $AB \ne \emptyset$ for a unique $B \in \cA$, $B \ne A$.

Since $|V(A)| \ge n-k \ge k+1$, every edge of $K$ not in $V(A)$ must share an end with an edge in $V(A)$, so every edge of $K$ not in $V(A)$ must be incident with $B$.  Thus, $V(D) = E(K) = \only A \cup AB \cup B\ov A$.
Since $K$ is connected and $|\cA| \ge 3$, there must be $C \in \cA$ adjacent to $B$ in $K$.  Since $v \in BC$ is nonadjacent to at most $k$ other edges of $K$, $|\only A| \le k$ and so $|V(B)| \ge n-k$.
The right side of Figure \ref{touchgraph34} shows the structure of $K$ in this situation.
We blow up $C$ using $B$ to give $\Phi'$ containing antifaces $A, B', C'$ with $|V(A)| \ge n-k$ and $|V(B')|, |V(C')| \ge 2k+3$.
If $\Phi'$ is not locally irreducible we apply Case \ref{reduce}, and otherwise we have interlacing by the Big and Moderate Faces Corollary (\ref{BMFlemma}) and use the Interlaced Faces Lemma (\ref{interlacing}), reducing the number of antifaces in either case.
\end{subsubcase}

\begin{subsubcase}\label{finalcase}
Suppose there is exactly one vertex of $K$ with one or more incident loops, and this vertex has at least two neighbors in $K$.
In other words, $\only A \ne \emptyset$ for a unique $A \in \cA$, and there are at least two $P \ne A$ with $AP \ne \emptyset$.
This is our final case, so if we can modify the embedding and obtain an earlier case, we are finished.

Choose $B \in \cA$ with $AB \ne \emptyset$.  Blow up $B$ using $A$, obtaining $\Phi'$ (and touch graph $K'$) with new antifaces $A', B'$ with $|V(A')|, |V(B')| \ge 2k+3$.
If $\Phi'$ is not described by Case \ref{finalcase}, we can apply a previous case to reduce the number of antifaces.  Thus, we may assume $\Phi'$ is locally irreducible, there is exactly one vertex of $K'$ with one or more incident loops, and this vertex has at least two neighbors in $K'$.  The vertex of $K'$ with incident loops must be one of $A'$ or $B'$, say $A'$, because blowing up $B$ using $A$ cannot create loops at other vertices of $K$. Thus $|V(A')| \geq n-k$.  Let $C$ be a neighbor of $A'$ other than $B'$.  Blow up $C$ using $A'$, obtaining $\Phi''$ (and touch graph $K''$) with antifaces $A'', C''$ with $|V(A'')|, |V(C'')| \ge 2k+3$.  Again, we may assume that $\Phi''$ is described by Case \ref{finalcase}, so there is exactly one vertex of $K''$ with one or more incident loops.
This must be one of $A''$ or $C''$, say $A''$. 
Therefore, $\Phi''$ has antifaces $A'', B', C''$ with $|V(A'')| \ge n-k$ and $|V(B')|, |V(C'')| \ge 2k+3$.  Hence, there is interlacing by the Big and Moderate Faces Corollary (\ref{BMFlemma}), and we can use the Interlaced Faces Lemma (\ref{interlacing}) to reduce the number of antifaces.
\end{subsubcase}
\end{subcase}
\end{case}

Thus, we can always reduce $|\cA|$ if $|\cA| \ge 3$, and by repeated reductions we obtain an embedding $\Phi_0$ with $|\cA| = 1$ or $2$ antifaces and with $\cC$ as the set of profaces.

We now verify that the embedding we have obtained satisfies all the conclusions of the theorem.
Consider the given digraph $D$ and the given directed circuit decomposition $\cC$, and take any directed embedding of $D$ in which $\cC$ is a subset of the faces.  By Observation \ref{diremb-relcd} this is orientable, we can choose an orientation so that $\cC$ is the collection of profaces, and if $\cW$ is the collection of other faces then $|\cW|$ has the same parity as $|V(D)|+|A(D)|+|\cC|$.
Hence, $|\cW| \ge 1$ or $2$ depending on whether $|V(D)+|A(D)|+|\cC|$ is odd or even, respectively.

The embedding $\Phi_0$ is an example of an embedding described in the previous paragraph, with $\cW = \cA$, and it achieves the lower bound on $|\cW|$.  It therefore minimizes the number of faces, and hence maximizes the genus, among all directed embeddings of $D$ in which all elements of $\cC$ are faces.
\end{proof}

The proof of Theorem \ref{densemain} describes a procedure for reducing the number of antifaces in an embedding, until only one or two remain.  It is easy to derive an algorithm from this proof.
The key manipulation of embeddings that we use is the Three Face Lemma (\ref{3face}), which can be implemented very simply using the rotation scheme representation of an embedding.
Almost all of the steps in the proof, and the supporting results in Sections \ref{sec:term} and \ref{SupLemmas}, involve explicit constructions that are easily implemented in polynomial time.
The one possible exception is our use of the Bipartite Degeneracy Lemma (\ref{bipdegen}): we use existence of a subgraph of minimum degree at least $d+1$ in a graph with sufficiently many edges.  However, it is not difficult to make this constructive: if a graph is known to be not $d$-degenerate, then a nonempty subgraph with minimum degree at least $d+1$ can be obtained by repeatedly deleting vertices of degree at most $d$ until it is impossible to continue.
Our proof of Theorem \ref{densemain} therefore yields a polynomial time algorithm to find the promised embedding.

\section{Conclusion}\label{sec:conclusion}

Naturally, we would like to improve our main result by weakening the required degree condition.  This seems difficult to do using our current arguments.  The hypothesis $\delta(\usg{D}) \ge (4n+2)/5$, or equivalently $n \ge 5k+7$, is used in a tight way in many places in our proof.  Every time we use the Blow Up Lemma (\ref{blowup}) to create two moderate-sized faces so we can apply the Big and Moderate Faces Corollary (\ref{BMFlemma}), in Cases \ref{noloopstar} and \ref{loop} of the proof of Theorem \ref{densemain}, we require $n \ge 5k+7$ to guarantee that the moderate-sized faces have at least $2k+3$ vertices.  We also require $n \ge 5k+7$ for Case \ref{noloopnostar}: Case \ref{mediumAB} requires $n \ge 5k+7$ to handle the situation when $|AB|=3k+2$, and we cannot extend the arguments in the other subcases of Case \ref{noloopnostar} to handle this situation.

However, we have no evidence to suggest that the degree condition in Theorem \ref{densemain} is best possible.  There may be a weaker condition that allows us to construct an oriented directed embedding of an eulerian digraph $D$ with a specified directed circuit decomposition $\cC$ as the profaces and just one or two antifaces.  In particular, this condition might give Theorem \ref{dirlatin} on latin squares as a corollary. The condition may even be much weaker, involving only a constant lower bound on minimum degree (of $D$ or $\usg{D}$), possibly combined with a connectivity condition such as a constant lower bound on edge-connectivity.  We only have very weak restrictions on what these bounds could be.  Examples discussed in \cite[Section 3]{EE-M2mod4} show that there are $4$-edge-connected $4$-regular (non-simple) graphs with no bi-eulerian embedding, and examples from \cite{EGS20} and \cite[Section 5]{EE-M2mod4} show that there are $4$-edge-connected $6$-regular graphs with a triangular decomposition that cannot be completed to an orientable embedding by adding at most two more faces.  Digraph examples can be derived from these by suitably directing the edges.

For further discussion of related problems, such as determining maximum genus for orientable directed embeddings of a given eulerian digraph, we refer the reader to \cite{EE-M2mod4}.

We close by noting that the general problem of finding spanning walks in graphs subject to various restrictions arises in many settings. Our work here was motivated by DNA self-assembly problems, but protein self-assembly problems, which are also related to graph embeddings but require different conditions, have also been studied, for example in \cite{FPR14}.  In addition to these applications involving biological structures, and to illustrate the range of these spanning walk problems, there is also the recent work \cite{DKL24, DKLpre}.  These papers give more variations, such as allowing edges to be repeated multiple times in any direction with different turning constraints from the topologically motivated ones given here, with application to robotics.   We expect that the interplay among these many and varied settings and problems will prove both interesting and productive.

\printbibliography

\end{document}